\documentclass[a4paper,11pt]{amsart}

\usepackage[french, english]{babel} 
\usepackage[T1]{fontenc}
\usepackage[ansinew]{inputenc}

\date{}

\title[Quantitative spectral inequalities]{Quantitative spectral inequalities for the anisotropic Shubin operators
and applications to null-controllability}

\author{Paul Alphonse}
\address{(Paul Alphonse) Universit\'e de Lyon, ENSL, UMPA - UMR 5669, F-69364 Lyon}
\email{paul.alphonse@ens-lyon.fr}

\author{Albrecht Seelmann}
\address{(Albrecht Seelmann) Technische Universit\"at Dortmund, Fakult\"at f\"ur Mathematik, D-44221 Dortmund, Germany}
\email{albrecht.seelmann@mathematik.tu-dortmund.de}

\makeatletter
	\@namedef{subjclassname@2020}{\textup{2020} Mathematics Subject Classification}
\makeatother


\usepackage[top=3cm, bottom=2cm, left=3cm, right=3cm]{geometry}	

\usepackage{enumitem}

\usepackage{array}
\usepackage{amsfonts}
\usepackage{amsmath}
\usepackage{amsthm}
\usepackage{amssymb}
\usepackage{mathtools}

\usepackage{multicol}

\usepackage{bbm}

\usepackage{stmaryrd}

\usepackage[scr]{rsfso}

\usepackage{comment}

\usepackage[dvipsnames]{xcolor}

\usepackage{hyperref}

\hypersetup{	
colorlinks=true,
breaklinks=true,
urlcolor= red,
linkcolor= red,
citecolor=Blue
}

\usepackage[textsize=tiny]{todonotes}

\numberwithin{equation}{section}

\newtheorem{thm}{Theorem}[section]
\newtheorem{prop}[thm]{Proposition}

\newtheorem{lem}[thm]{Lemma}
\newtheorem{cor}[thm]{Corollary}
\theoremstyle{definition}
\newtheorem{dfn}[thm]{Definition}
\newtheorem{ex}[thm]{Example}
\newtheorem{rk}[thm]{Remark}

\DeclareMathOperator{\Ran}{Ran}
\DeclareMathOperator{\diam}{diam}
\DeclareMathOperator{\dist}{dist}
\DeclareMathOperator{\spec}{spec}

\DeclarePairedDelimiter{\abs}{\lvert}{\rvert}
\DeclarePairedDelimiter{\norm}{\lVert}{\rVert}
\DeclarePairedDelimiter{\sprod}{\langle}{\rangle}

\newcommand{\RR}{\mathbb{R}}
\newcommand{\ZZ}{\mathbb{Z}}
\newcommand{\NN}{\mathbb{N}}
\newcommand{\TT}{\mathbb{T}}

\newcommand{\cB}{\mathcal{B}}
\newcommand{\cE}{\mathcal{E}}
\newcommand{\cF}{\mathcal{F}}
\newcommand{\cO}{\mathcal{O}}

\newcommand{\euler}{e}
\newcommand{\scrit}{\zeta}

\newcommand{\bmone}{\mathbbm{1}}
\newcommand{\power}{\eta}

\newcommand{\dd}{\mathrm d}
\newcommand{\obs}{\mathrm{obs}}

\begin{document}

\begin{abstract}
	We prove quantitative spectral inequalities for the (anisotropic) Shubin operators on the whole
	Euclidean space, thus relating for functions from spectral subspaces associated to finite energy intervals their $L^2$-norm on
	the whole space to the $L^2$-norm on a suitable subset. A particular feature of our estimates is that the constant relating
	these $L^2$-norms is very explicit in geometric parameters of the corresponding subset of the whole space, which may become
	sparse at infinity and may even have finite measure. This extends results obtained recently by J.\ Martin
	and, in the particular case of the harmonic oscillator, by A.\ Dicke, I.\ Veseli\'c, and the second author.
	We apply our results towards null-controllability of the associated parabolic equations, as well as to the ones
	associated to the (degenerate) Baouendi-Grushin operators acting on $\RR^d \times \TT^d$.
\end{abstract}

\selectlanguage{english}

\maketitle

\section{Introduction}

Quantitative spectral inequalities are instances of so-called \emph{uncertainty relations} that, in the context of the present
paper, take the form
\[
	\norm{f}_{L^2(\Omega)}^2
	\leq
	d_0 e^{d_1\lambda^{\power}} \norm{f}_{L^2(\omega)}^2
	,\quad
	f \in \cE_\lambda(A)
	,\
	\lambda \geq 0
	,
\]
where $\omega$ is a measurable subset of a domain $\Omega \subset \RR^d$, $\cE_\lambda(A) = \bmone_{(-\infty,\lambda]}(A)$
denotes the spectral subspace for a non-negative selfadjoint operator $A$ in $L^2(\Omega)$ associated with the interval
$(-\infty,\lambda]$, and $d_0,d_1,\power > 0$ are constants. Such inequalities can be viewed as quantitative variants of an
identity theorem (in the sense that $f = 0$ on $\omega$ implies $f = 0$ on $\Omega$) and are often considered under different
names, depending on the context, such as \emph{(quantitative) unique continuation estimates}, see e.g.
\cite{RousseauL-12,LogunovM-20}, or \emph{uncertainty principles}, see e.g. \cite{StollmannS-21}. The notion
\emph{spectral inequality} we adopt is common in the context of control theory, see e.g., \cite{LaurentL-21, RousseauL-12}.
They are also closely related to the so-called \emph{vanishing order}, see, e.g., \cite{DonnellyF-88,LaurentL-21}, and
\emph{annihilating pairs} in Fourier analysis, see e.g. \cite{BJKPS, HavinJ-94}.

In the present work, we prove spectral inequalities from sparse sensor sets $\omega$ with an explicit form of the constants
when $A$ is the (anisotropic) Shubin operator in $L^2(\RR^d)$,
\[
	H_{k,m}
	=
	(-\Delta)^m + \abs{x}^{2k}
	,\quad
	x \in \RR^d
	,
\]
where $k,m \geq 1$ are positive integers.
Our inequalities complement recent results from \cite{Martin} and, in the particular case of the harmonic oscillator,
from \cite{DickeSV-23}. For instance, very general spectral inequalities have been obtained in \cite[Theorem~2.1\,(ii)]{Martin} for
every measurable set $\omega \subset \RR^d$ with merely positive measure. These inequalities take the form
\begin{equation}\label{eq:genesti}
	\norm{f}^2_{L^2(\RR^d)}
	\leq
	Ke^{K\lambda^{\frac1{2k}+\frac1{2m}}\abs{\log\lambda}} \norm{f}^2_{L^2(\omega)}
	,\quad
	f \in \cE_\lambda(H_{k,m})
	,\
	\lambda>0
	,
\end{equation}
where $K>0$ is a positive constant depending on $k$, $m$, the dimension $d$ and the set $\omega$. The dependence of $K$ on the
set $\omega$, however, is not explicit, even if more information on $\omega$ is available. Our inequalities mainly address this
dependence if $\omega$ is sparse in a sense made precise below. The technique of proof used in the present paper follows the
approach by Kovrijkine \cite{Kovrijkinethesis,Kovrijkine} and builds upon recent developments in this field of research
\cite{BJKPS,DickeSV-23,EgidiS-21,Martin,MPS}. We apply our results in the context of exact null-controllability for the abstract
Cauchy problems associated to $H_{k,m}$, as well as to the Baouendi-Grushin operator in
$L^2(\RR^d \times \TT^d)$,
\[
	\Delta_{\gamma}
	=
	\Delta_x + \abs{x}^{2\gamma}\Delta_y
	,\quad
	(x,y) \in \RR^d \times \TT^d
	,
\]
with $\gamma\geq1$ a positive integer. Note that for the latter we use the more traditional parameter $\gamma$, rather
than just $k$ as for the Shubin operators.

\subsubsection*{Outline of the work}
In Section \ref{sec:results}, we present in detail the main results contained in this work.
Section~\ref{sec:spectralEst} is then devoted to the proof of the spectral inequalities for the
anisotropic Shubin operators.
These spectral inequalities are used in Section~\ref{sec:cont} to prove null-controllability results for
the evolution equations associated with both the Shubin operators on $\RR^d$ and the Baouendi-Grushin operators on
$\RR^d \times \TT^d$. Finally, Appendix~\ref{sec:asymptotics} provides a statement on the asymptotics of the smallest eigenvalue
of the anisotropic Shubin operator $H_{k,1}$ as $k \to \infty$, which is used in Example~\ref{ex:example}.

\subsubsection*{Notations} The following notations and conventions will be used throughout this work:
\begin{enumerate}[label={\arabic*.},leftmargin=* ,parsep=2pt,itemsep=0pt,topsep=2pt]
	\item $\NN$ denotes the set of natural numbers starting from zero.
	\item The canonical Euclidean scalar product of $\mathbb R^d$ is denoted by $\cdot$, and $\vert\cdot\vert$ stands for the
				associated canonical Euclidean norm. We will also use the Japanese bracket notation $\langle\cdot\rangle = (1+\vert\cdot\vert^2)^{1/2}$.
	\item The length of any multi-index $\alpha=(\alpha_1,\cdots,\alpha_d)\in\mathbb N^d$ is denoted $\vert\alpha\vert$ and defined by
	$$\vert\alpha\vert = \alpha_1+\cdots+\alpha_d.$$
	\item The Lebesgue measure of a measurable set $\omega\subset\mathbb R^d$ is denoted $\vert\omega\vert$.
	\item $\mathbbm1_{\omega}$ denotes the characteristic function of any subset $\omega\subset\mathbb R^d$.
	\item For all measurable subsets $\omega\subset\mathbb R^d$, the inner product of $L^2(\omega)$ is denoted
				$\langle\cdot,\cdot\rangle_{L^2(\omega)}$, while $\Vert\cdot\Vert_{L^2(\omega)}$ stands for the associated norm.
	\item For a nonnegative selfadjoint operator $A$ on $L^2(\RR^d)$, $\cE_\lambda(A) = \bmone_{(-\infty,\lambda]}(A)$ with
				$\lambda \geq 0$ denotes the spectral subspace for $A$ associated with the interval $(-\infty,\lambda]$.
\end{enumerate}

\subsubsection*{Acknowledgments}
The first author thanks J.~Martin for many enthusiastic discussions during the preparation of this work and for pointing out
some relevant references. He also thanks J.F.~Bony for very fruitful discussions on the theory of large coupling limit and for
pointing out many references on this topic. He finally warmly thanks I.~Veseli\'c and A.~Seelmann for their kind hospitality at
the TU Dortmund, where this work was initiated. The second author has been partially supported by the DFG grant VE 253/10-1
entitled \textit{Quantitative unique continuation properties of elliptic PDEs with variable 2nd order coefficients and
applications in control theory, Anderson localization, and photonics.}

\section{Statement of the main results}\label{sec:results}

This section is devoted to present in detail the main results contained in this work.

\subsection{Spectral inequalities for the Shubin operators}
Given two positive integers $k,m \geq 1$, we consider in $L^2(\RR^d)$ the (anisotropic) Shubin operator
\[
	H_{k,m}
	=
	(-\Delta)^m + \abs{x}^{2k}
	,\quad
	x \in \RR^d
	,
\]
which is a non-negative and selfadjoint operator with purely discrete spectrum when equipped with its maximal domain
\[
	D(H_{k,m})
	=
	\big\{ g \in L^2(\RR^d) \colon H_{k,m}g \in L^2(\RR^d) \big\}
	.
\]
Moreover, for $\lambda\geq0$, let $\mathcal E_{\lambda,k,m} = \cE_\lambda(H_{k,m}) = \Ran\bmone_{(-\infty,\lambda]}(H_{k,m})$
denote the spectral subspace for the operator $H_{k,m}$ associated with the interval $(-\infty,\lambda]$.

For easier comparison, let us first state a result for the harmonic oscillator, corresponding to the case where $k=m=1$, which
covers and extends previous results from \cite{BJKPS,DickeSV-23,EgidiS-21,MPS}, see Remark~\ref{rk:mainharmo} below.

\begin{thm}\label{thm:mainharmo}
	Let $\rho \colon \RR^d \to (0,+\infty)$ and $\sigma \colon \RR^d \to (0,1]$ be functions such that $\rho$ and $1/\sigma$ are
	locally bounded, and let $\omega \subset \RR^d$ be a measurable set satisfying
	\begin{equation}\label{eq:omega}
		\forall x\in\RR^d
		,\quad
		\abs{\omega\cap B(x,\rho(x))} \geq \sigma(x)\abs{B(x,\rho(x))}
		.
	\end{equation}
	Then, there exists a positive constant $K > 0$, depending only on the dimension $d$, such that for all $\lambda\geq0$ and
	$f\in\mathcal E_{\lambda,1,1}$ we have
	\begin{equation}\label{eq:specIneqharmo}
		\norm{f}_{L^2(\RR^d)}^2
		\leq
		\biggl( \frac K{\theta_\lambda} \biggr)^{K( 1+(L_\lambda)^2 + L_\lambda\sqrt{\lambda})}
			\norm{f}_{L^2(\omega)}^2
		,
	\end{equation}
	where
	\[
		\theta_\lambda
		:=
		\inf_{\abs{x} < \sqrt{2\lambda}} \sigma(x)
		\quad\text{ and }\quad
		L_\lambda
		:=
		\sup_{\abs{x} < \sqrt{2\lambda}} \rho(x)
		.
	\]
\end{thm}

\begin{rk}\label{rk:mainharmo}
	Suppose that the functions $\sigma$ and $\rho$ satisfy the bounds
	\begin{equation}\label{eq:gammarhoBounds}
		\forall x \in \RR^d
		,\quad
		\sigma(x) \geq \theta^{\sprod{x}^a}
		\quad\text{ and }\quad
		\rho(x) \leq L\sprod{x}^\delta
	\end{equation}
	with some fixed $\theta \in (0,1]$, $a \geq 0$, $L > 0$, and $\delta \geq 0$. In this case, we have
	\[
		\theta_\lambda
		\geq
		\theta^{(1+2\lambda)^{a/2}}
		\quad\text{ and }\quad
		L_\lambda
		\leq
		L(1+2\lambda)^{\delta/2}
		.
	\]
	It is then straightforward to verify that \eqref{eq:specIneqharmo} takes the form
	\begin{equation}\label{eq:specIneqharmoGeom}
		\norm{f}_{L^2(\RR^d)}^2
		\leq
		\biggl( \frac K{\theta} \biggr)^{K^{1+a+\delta}(1+L^2\lambda^{\delta+a/2} + L\lambda^{(1 + a+\delta)/2})}
			\norm{f}_{L^2(\omega)}^2
		,\quad
		f \in \cE_{\lambda,1,1}
		,
	\end{equation}
	with a possibly different constant $K \geq 1$. This covers \cite[Theorem~2.7]{DickeSV-23}, while the particular case of $a = 0$ has
	also previously been considered in \cite[Theorem~2.1]{MPS} under the additional assumption that $\rho$ is $1/2$-Lipschitz
	continuous.

	The case where the functions $\sigma$ and $\rho$ are constant, and thus the parameters $a$ and $\delta$ above can be chosen
	equal to zero, that is,
	\begin{equation}\label{eq:thick}
		\forall x\in\mathbb R^d
		,\quad
		\abs{ \omega \cap B(x,L) } \geq \theta \abs{ B(x,L) }
		,
	\end{equation}
	corresponds to so-called \emph{$(\theta,L)$-thick} sets. Such sets have been getting considerable attention in the past and have
	been previously discussed in this context in \cite[Theorem~2.1\,(iii)]{BJKPS} and \cite[Corollary~1.9]{EgidiS-21}. In fact,
	\cite[Corollary~1.9]{EgidiS-21} also makes in this case the dependence on the dimension in \eqref{eq:specIneqharmoGeom}
	explicit. This could have been done in \eqref{eq:specIneqharmo} with our technique as well, but we refrained from doing so for
	the sake of simplicity.
\end{rk}

The spectral inequality in \eqref{eq:specIneqharmo} is very explicit in terms of $\sigma$ and $\rho$. The fact that only the
uniform bounds of $\sigma$ and $\rho$ on the ball $B(0,\sqrt{2\lambda})$ enter the estimate \eqref{eq:specIneqharmo} is due to
the strong decay that the potential enforces on the eigenfunctions of the harmonic oscillator (and finite linear combinations
thereof). This is an instance of a much more general phenomenon that also takes place in case of general (anisotropic) Shubin
operators and eventually leads to a variant of Theorem~\ref{thm:mainharmo} for these operators that, in particular, gives a
positive answer to \cite[Conjecture~1.6]{DickeSV2}. Our corresponding main result considers exactly the same geometry for
$\omega \subset \RR^d$ as in Theorem \ref{thm:mainharmo} and reads as follows.

\begin{thm}\label{thm:main}
	There exists a constant $K > 0$, depending only on $k$, $m$, and the dimension $d$, such that for all measurable sets
	$\omega \subset \RR^d$ satisfying the geometric condition \eqref{eq:omega}, and all $\lambda \geq 0$ and
	$f \in \cE_{\lambda,k,m}$ we have
	\begin{equation}\label{eq:specIneq}
		\norm{f}_{L^2(\RR^d)}^2
		\leq
		\biggl( \frac K{\theta_{\lambda,k}} \biggr)^{K(1+(L_{\lambda,k})^{1+\frac km} + L_{\lambda,k}
			\lambda^{\frac1{2m}} + \log(1+\lambda))}\norm{f}_{L^2(\omega)}^2
		,
	\end{equation}
	where
	\begin{equation}\label{eq:geom}
		\theta_{\lambda,k}
		:=
		\inf_{\abs{x} < (2\lambda)^{1/2k}} \sigma(x)
		\quad\text{ and }\quad
		L_{\lambda,k}
		:=
		\sup_{\abs{x} < (2\lambda)^{1/2k}} \rho(x)
		.
	\end{equation}
\end{thm}

\begin{rk}\label{rk:badterm}
	Similarly as for the harmonic oscillator, the potential $\abs{x}^{2k}$ enforces a strong decay of (finite linear combinations
	of) eigenfunctions of the	operator $H_{k,m}$, so that such functions are localized around the origin. More precisely,
	Corollary~\ref{cor:locBernstein} below states that for all $\lambda \geq 0$ and $f \in \cE_{\lambda,k,m}$,
	\[
		\norm{f}_{L^2(\RR^d)}^2
		\leq
		2\norm{f}_{L^2(B(0,(2\lambda)^{1/2k}))}^2
		.
	\]
	It is therefore sufficient to prove for functions in $\cE_{\lambda,k,m}$ estimates on the ball $B(0,(2\lambda)^{1/2k})$ in order
	to obtain similar estimates on the whole space $\RR^d$. This also explains why in \eqref{eq:specIneq} only the bounds of
	$\sigma$ and $\rho$ on the ball $B(0,(2\lambda)^{1/2k})$ enter.

	While the just mentioned localization behaviour is completely consistent with the case of the harmonic oscillator in
	Theorem~\ref{thm:mainharmo}, it is worth to note that the term $\log(1+\lambda)$ on the right-hand side of \eqref{eq:specIneq}
	does not appear in \eqref{eq:specIneqharmo}. This term turns out to be quite unfavourable (see Remark~\ref{rk:mainComparison}
	below), and we conjecture that it can indeed be just skipped. The reason why it comes into play within our framework is related
	to obtaining Agmon estimates for spectral subspaces as explained in Remark~\ref{rk:agmon} in Section~\ref{ssec:Agmon} below.
	Nevertheless, since $\log(1+\lambda)$ is dominated by every power of $\lambda$, it should be emphasized that our bound from
	Theorem~\ref{thm:main} still gives a proper quantitative spectral inequality that is strong enough to be applied in the context
	of null-controllability and thus obtain results in Corollaries~\ref{cor:fractionalShubin} and~\ref{cor:contlog} and
	Theorems~\ref{thm:bagrushincont} and~\ref{thm:critdiss}\,(ii) below that were otherwise not accessible before.
\end{rk}

\begin{rk}\label{rk:mainComparison}
	Suppose again that the functions $\sigma$ and $\rho$ satisfy \eqref{eq:gammarhoBounds}, so that
	\[
		\theta_{\lambda,k}
		\geq
		\theta^{(1+(2\lambda)^{1/k})^{a/2}}
		\quad\text{ and }\quad
		L_{\lambda,k}
		\leq
		L(1+(2\lambda)^{1/k})^{\delta/2}
		.
	\]
	In this case, it is easy to check that the spectral inequality \eqref{eq:specIneq} can be written as
	\begin{equation}\label{eq:specweakly}
		\norm{f}_{L^2(\RR^d)}^2
		\leq
		\biggl(\frac K{\theta}\biggr)^{K^{1+a+\delta} (1+\lambda^{\frac{a}{2k}}) (1+L^{1+\frac km}\lambda^{\delta(\frac1{2k} +
			\frac1{2m})} + L\lambda^{\frac{\delta}{2k} + \frac1{2m}} + \log(1+\lambda))}\norm{f}_{L^2(\omega)}^2
	\end{equation}
	with a possibly different constant $K \geq 1$. This extends \cite[Theorem~2.1\,(i)]{Martin}, where only the case $a = 0$ and
	$\delta \in [0,1]$ is considered. At the same time, our bound in \eqref{eq:specweakly} is much more explicit in the model
	parameters, which is very useful in the context of control theory, see Section~\ref{ssec:exactContr} below. It should be
	mentioned, however, that in \eqref{eq:specweakly} with $a = 0$ the formal homogenization limit as $L \to 0$ results in a
	right-hand side where the constant still depends on $\lambda$. This is due to the $\log(1+\lambda)$-term in
	\eqref{eq:specweakly} (resp.\ \eqref{eq:specIneq}) but is highly unintuitive and not consistent with the known behaviour for the
	free Laplacian and the harmonic oscillator. This is one reason why this term is considered unfavourable and should be removed in
	future research if possible, cf.\ Remark~\ref{rk:agmon} below.

	It is also worth to note that for $a = 0$ (for simplicity) and $\delta \in [0,1]$ the estimate \eqref{eq:specweakly} can for
	$\lambda \geq 1$ be written as
	\[
	\norm{f}^2_{L^2(\RR^d)}
	\leq
	Ke^{K\lambda^{\frac{\delta}{2k}+\frac1{2m}}} \norm{f}^2_{L^2(\omega)}
	\]
	with yet another constant $K > 0$, now also depending on $L$, $\theta$, and $\delta$. This is stronger than the general estimate
	\eqref{eq:genesti}. By contrast, if $a = 0$ and $\delta > 1$, estimate \eqref{eq:specweakly} writes for $\lambda \geq 1$ as
	\[
		\norm{f}^2_{L^2(\RR^d)}
		\leq
		Ke^{K\lambda^{\delta(\frac1{2k}+\frac1{2m})}} \norm{f}^2_{L^2(\omega)}
	\]
	and is therefore worse than the general estimate \eqref{eq:genesti}, although the latter only uses that $\omega$ has positive
	measure. It is not yet clear how to reconcile this different behavior in the two regimes $\delta \leq 1$ and $\delta > 1$.
\end{rk}

In the end of this subsection, let us present examples of measurable sets satisfying the geometric condition \eqref{eq:omega}.

\begin{ex}
	Suppose that the local scale $\rho \equiv L > 0$ is constant and that	$\sigma = w/(\sqrt{d}+1)^d$ with a radially symmetric
	function $w \colon \RR^d \to (0,1]$ that is non-increasing with respect to the modulus and for which $1/w$ is locally bounded.
	Inspired by \cite[Example~2.3]{DickeSV-23} and \cite[Example~4.17]{Dickethesis}, with $l = L/(\sqrt{d}+1)$ and
	$r_j = l w(j)^{1/d}$ consider the set
	\[
		\omega
		=
		\bigcup_{j \in l\ZZ^d} B(j,r_j)
		.
	\]
	This set $\omega$ satisfies the geometric condition \eqref{eq:omega}. Indeed, given $x \in \RR^d$, there is $j \in l\ZZ^d$ with
	$\abs{j} \leq \abs{x}$ and $\abs{x - j} < l\sqrt{d}$, so that $\abs{x - j} + r_j < l(\sqrt{d}+1) = L$. Hence, the ball $B(x,L)$
	contains the ball $B(j,r_j)$, so that
	\[
		\frac{\abs{\omega \cap B(x,L)}}{\abs{B(x,L)}}
		\geq
		\frac{\abs{B(j,r_j)}}{\abs{B(x,L)}}
		=
		\biggl( \frac{r_j}{L} \biggr)^d
		=
		\sigma(j)
		\geq
		\sigma(x)
		.
	\]
	It is worth to note that under the condition $\sum_{j \in l\ZZ^d} w(j) < \infty$, the above set $\omega$ has finite measure.
\end{ex}

\begin{ex}
	Suppose that $d \geq 2$. Inspired by \cite[p.\ 32]{Martinthesis}, let us consider a non-decreasing continuous function
	$R \colon [0,+\infty) \to (0,+\infty)$, a non-increasing continuous function $r \colon [0,+\infty) \to (0,1)$, and the
	associated set
	\[
		\omega_{r,R}
		=
		\bigl\{ (x,y) \in \RR^{d-1} \times \RR \colon \abs{y} > R(\abs{x})(1-r(\abs{x})) \bigr\}
		.
	\]
	It is then easy to see that the intersection $\omega_{r,R} \cap B((x,0),R(|x|))$ is always non-empty (and open). Hence, the set
	$\omega_{r,R}$ satisfies the geometric condition \eqref{eq:omega} with the functions $\rho$ and $\sigma$ given by
	$\rho(x,y) = R(\abs{x})$ and
	\[
		\sigma(x,y)
		:=
		\frac{\abs{\omega_{r,R}\cap B((x,y),\rho(x,y))}}{\abs{B((x,y),\rho(x,y))}}
		\geq
		\frac{\abs{\omega_{r,R}\cap B((x,0),R(\abs{x}))}}{\abs{B((x,0),R(\abs{x}))}}
		>
		0
		,
	\]
	respectively.
\end{ex}

\subsection{Exact null-controllability}\label{ssec:exactContr}
As application of the spectral inequalities from Theorems~\ref{thm:mainharmo} and \ref{thm:main}, we study the exact
null-controllability for two classes of diffusive equations, being elliptic and hypoelliptic, respectively.

\begin{dfn}[Exact null-controllability]
	Let $\Omega \subset \RR^d$ be a domain, and let $P$ be a non-negative selfadjoint operator in $L^2(\Omega)$.
	Given a measurable set $\omega \subset \Omega$, the evolution
	equation
	\begin{equation}\label{eq:evoEq}
		\begin{cases}
			\partial_tf(t,x) + Pf(t,x) = h(t,x)\bmone_{\omega}(x), & t>0,\ x \in \Omega,\\
			f(0,\cdot) = f_0\in L^2(\Omega),
		\end{cases}
	\end{equation}
	is said to be \emph{exactly null-controllable from the control support $\omega$ in time $T > 0$} if for every initial datum
	$f_0\in L^2(\Omega)$ there exists a control function $h\in L^2((0,T)\times\Omega)$ such that the mild
	solution to \eqref{eq:evoEq} satisfies $f(T,\cdot) = 0$.
\end{dfn}

\subsubsection{The fractional anisotropic Shubin evolution equations}\label{sssec:fractionalShubin}
Let us first consider the evolution equations of the form \eqref{eq:evoEq} associated to the elliptic operators $P = H^s_{k,m}$
with $s > 0$, that is, 
\begin{equation}\label{eq:fractionalShubin}\tag{$E_{s,k,m}$}
	\begin{cases}
		\partial_tf(t,x) + H^s_{k,m}f(t,x) = h(t,x)\bmone_{\omega}(x), & t>0,\ x\in\RR^d,\\
		f(0,\cdot) = f_0\in L^2(\RR^d).
	\end{cases}
\end{equation}
Here, the fractional powers of the operator $H_{k,m}$ are understood via standard functional calculus.

The spectral inequalities in Theorems~\ref{thm:mainharmo} and~\ref{thm:main} allow us to derive many exact null-controllabil\-ity
results for the equation \eqref{eq:fractionalShubin}, and we choose to present only three statements. We first give two general
results closely related to Remark~\ref{rk:mainComparison}.

\begin{cor}\label{cor:fractionalShubin}
	Let $\omega \subset \RR^d$ be a measurable set as in \eqref{eq:omega}, and suppose that the two functions
	$\sigma \colon \RR^d \to (0,1]$ and $\rho \colon \RR^d \to (0,+\infty)$ satisfy
	\[
		\sigma(x)
		\geq
		\theta^{\sprod{x}^a}
		\quad\text{ and }\quad
		\rho(x)
		\leq
		L\sprod{x}^\delta
		,\quad
		x \in \RR^d
		,
	\]
	with some fixed $L > 0$, $\delta \in [0,1]$, $\theta \in (0,1]$, and $a \geq 0$. Then, for all $s > 0$ satisfying
	\[
		\frac{\delta+a}{2k} + \frac1{2m}
		<
		s
		,
	\]
	the equation \eqref{eq:fractionalShubin} is exactly null-controllable from $\omega$ in every positive time $T>0$.
\end{cor}

\begin{rk}
	Corollary~\ref{cor:fractionalShubin} extends \cite[Corollary 2.12]{Martin-22} (cf.\ also \cite[Corollary~1.2]{DickeS-22}), which
	only deals with the case $a = 0$. Moreover, recall from \cite[Theorem 2.5]{Martin} (whose proof is based on the general spectral
	inequalities \eqref{eq:genesti}) that whenever $s>1/(2k)+1/(2m)$, the equation \eqref{eq:fractionalShubin} is exactly
	null-controllable from every measurable control support $\omega \subset \RR^d$ with positive measure and in every positive time
	$T>0$. Corollary~\ref{cor:fractionalShubin} therefore provides a new result only in the case $0\le\delta+a < 1$.
\end{rk}

\begin{cor}\label{cor:contlog}
	Let $\omega \subset \RR^d$ be a measurable set as in \eqref{eq:omega}, where the function $\sigma$ satisfies
	\[
		\sigma(x)
		\geq
		\theta^{\sprod{x}^a}
		,\quad
		x \in \RR^d
		,
	\]
	with some fixed $\theta \in (0,1]$ and $a \geq 0$, and the function $\rho$ exhibits a growth at infinity that is slower
	than any power, that is,
	\[
		\forall \delta>0
		,\quad
		\rho(x) = o(\abs{x}^{\delta})\quad\text{as $\abs{x}\rightarrow+\infty$}
		.
	\]
	Then, for all $s > a/2k + 1/2m$, the equation \eqref{eq:fractionalShubin} is exactly null-controllable from the control support
	$\omega$ in every positive time $T>0$.
\end{cor}

\begin{rk}
	Corollary~\ref{cor:contlog} is, in fact, a quite straightforward consequence of Corollary~\ref{cor:fractionalShubin}, see
	Section~\ref{ssec:proofCorollariesShubin} below. Nevertheless, it should be mentioned that the particular case of $a = 0$,
	although not explicitly stated in the literature, could have been proven also by using the results from
	\cite[Chapter 6, Section 3]{Martinthesis}.
\end{rk}

It is well known from \cite[Theorem 1.10]{M} that the equation $($\hyperref[eq:fractionalShubin]{$E_{1,1,1}$}$)$ is not
null-controllable in any positive time whenever the control support $\omega \subset \RR^d$ is contained in a half space. In fact,
it can be readily checked that a half space satisfies a geometric condition of the form \eqref{eq:omega} with a constant function
$\sigma$ and a function $\rho$ taking the form
\[
	\rho(x)
	=
	L \sprod{x}
	,\quad
	x \in \RR^d,
\]
with some $L > 0$. Note that the latter exhibits a linear growth and is thus indeed excluded in
Corollaries~\ref{cor:fractionalShubin} and~\ref{cor:contlog} above. This, however, raises the question whether local scales
$\rho$ can be allowed that exhibit an arbitrary sublinear growth. A first step in this direction is taken by the following last
result of this subsection.

\begin{cor}\label{cor:harmolog}
	Let $\omega \subset \RR^d$ be a measurable set as in \eqref{eq:omega}, and suppose that the function $\sigma$ is
	constant and that $\rho$ satisfies
	\begin{equation}\label{eq:harmolog}
		\rho(x)
		\leq
		\frac{L\sprod{x}}{(g\circ g)^\alpha(\abs{x})g(\abs{x})}
		\quad\text{ where }\quad
		g(r)
		=
		\log(e+r)
		,\quad
		r \geq 0
		,
	\end{equation}
	with some $L > 0$ and $\alpha > 2$. Then, the equation $($\hyperref[eq:fractionalShubin]{$E_{1,1,1}$}$)$ is exactly
	null-controllable from the control support $\omega$ in every positive time $T>0$.
\end{cor}

\subsubsection{The Baouendi-Grushin heat equation}\label{subsubsec:grushin}

Let us now consider the fractional heat-like hypoelliptic evolution equation associated with the Baouendi-Grushin operator,
\begin{equation}\label{eq:bagrushin}\tag{$E_{\gamma,s}$}
	\begin{cases}
		\partial_tf(t,x,y) + (- \Delta_{\gamma})^sf(t,x,y) = h(t,x,y)\bmone_{\omega}(x,y),\quad t>0,\ (x,y)\in\RR^d\times\TT^d, \\
		f(0,\cdot,\cdot) = f_0\in L^2(\RR^d\times\mathbb T^d),
	\end{cases}
\end{equation}
where $s>0$ and $\gamma\geq1$ is a positive integer. Here, the Baouendi-Grushin operator $\Delta_{\gamma}$ acting on
$\RR^d \times \TT^d$,
\[
	\Delta_{\gamma}
	=
	\Delta_x + \abs{x}^{2{\gamma}}\Delta_y
	,\quad
	(x,y) \in \RR^d \times \TT^d,
\]
is equipped with its maximal domain, which makes it a positive selfadjoint operator. Note that the hypothesis that $\RR^d$ and
$\TT^d$ have the same spacial dimension $d$ is just for simplicity, and nothing substantial would change if different dimensions
would be allowed.

Our first result regarding the equation \eqref{eq:bagrushin} gives a necessary geometric condition on the control support $\omega$
for \eqref{eq:bagrushin} to be exactly null-controllable. It holds for all dissipation parameters $s > 0$.

\begin{prop}\label{prop:neccond}
	If the equation \eqref{eq:bagrushin} is exactly null-controllable from the control support $\omega \subset \RR^d \times \TT^d$,
	then there exist $L>0$ and $\theta \in (0,1]$ such that
	\begin{equation}\label{eq:thickRT}
		\forall x \in \RR^d
		,\quad
		\abs{ \omega \cap (B(x,L) \times \TT^d) }
		\geq
		\theta \abs{B(x,L)}
		.
	\end{equation}
\end{prop}

Positive null-controllability results for the equation \eqref{eq:bagrushin} strongly depend on how the dissipation parameter $s$
relates to the critical hypoelliptic parameter $(1+\gamma)/2$. Let us first state a precise characterisation of
null-controllability for a particular class of control supports in the strong dissipation regime $s > (1+\gamma)/2$.

\begin{thm}\label{thm:bagrushincont}
	Suppose that $s > (1+\gamma)/2$, and let $T > 0$ and $\omega \subset \RR^d$ be measurable. The
	following assertions are equivalent:
	\begin{enumerate}

		\item[$(i)$]
		The equation \eqref{eq:bagrushin} is exactly null-controllable from the control support $\omega \times \TT^d$ in time $T$.

		\item[$(ii)$]
		The set $\omega$ is thick in $\RR^d$.

	\end{enumerate}
\end{thm}

\begin{rk}\label{rk:bagrushincont}
	The techniques presented in the current work only allow to consider in the above result control supports that are strips of the
	form $\omega \times \TT^d$, but not more general control supports satisfying the condition \eqref{eq:thickRT}. The latter
	require a more sophisticated approach, which we postpone to a follow-up paper \cite{AS}. In particular, we prove there that
	Theorem~\ref{thm:bagrushincont} also holds for such more general control supports.
\end{rk}

In the critical dissipation regime $s = (1+\gamma)/2$, we state a positive null-controllability results from strips,
and also a negative one for control supports avoiding the degeneracy line $\{x = 0\}$.

\begin{thm}\label{thm:critdiss}
	Suppose that $s = (1+\gamma)/2$, and denote by $\lambda_{\gamma}>0$ the smallest eigenvalue of the anharmonic
	oscillator $H_{\gamma,1}$.
	\begin{enumerate}

		\item[$(i)$]
		For every measurable set $\omega \subset \RR^d \times \TT^d$ satisfying the condition
		$\overline{\omega} \cap \{x=0\} = \emptyset$, the equation \eqref{eq:bagrushin} is never exactly null-controllable from $\omega$
		in time $T>0$ when $0 <T < T_*$, where the time $T_*>0$ is given by
		\[
			T_*
			=
			\frac1{1+\gamma} \biggl( \frac{\dist(0,\omega)}{\sqrt{\lambda_{\gamma}}} \biggr)^{1+\gamma}
			.
		\]

		\item[$(ii)$]
		There exists a positive constant $c_{\gamma} > 0$ such that for every $(\theta,L)$-thick set $\omega \subset \RR^d$,
		the equation \eqref{eq:bagrushin} is exactly null-controllable from the control support $\omega \times \TT^d$ in every
		positive time $T \geq T^*$, where $T^*>0$ is given by
		\[
			T^*
			=
			c_{\gamma}\,\biggl( \frac L{\sqrt{\lambda_{\gamma}}} \biggr)^{1+\gamma}\,\log\Bigl(\frac{c_{\gamma}}{\theta}\Bigr)
			.
		\]

	\end{enumerate}
\end{thm}

\begin{rk}\label{rk:cridiss}
	(1)
	Recall from \cite[Theorem 4.12]{K} that when $d=1$ and $\gamma = s = 1$, the equation (\hyperref[eq:bagrushin]{$E_{1,1}$}) is never
	exactly null-controllable from any control support of the form $\RR \times \omega$ where $\omega = \TT \setminus [a,b]$. Therefore,
	one does not expect positive null-controllability results to hold for the equation \eqref{eq:bagrushin}
	in the regime $s = (1+\gamma)/2$ from more general control supports $\omega \subset \RR^d \times \TT^d$ satisfying the
	condition \eqref{eq:thickRT}.
	
	(2)
	Part~($i$) of the above statement is consistent with known results from the literature for the particular case $\gamma = 1$.
	Indeed, the time $T_*$ then reduces to
	\[
		T_*
		=
		\frac{\dist(0,\omega)^2}{2d}
		,
	\]
	and therefore takes the very same form as the (minimal) times appearing in the study of Grushin-type models, see, e.g.,
	\cite[Theorem 1.1]{ABM}, \cite[Theorem 1]{BCG}, \cite[Theorem 1.3]{BMM}, or \cite[Theorem 1.1]{BDE}.

\end{rk}

Our last result considers control supports $\omega \subset \RR^d \times \TT^d$ avoiding the degeneracy line $\{x=0\}$ in the
weak dissipation regime.

\begin{thm}\label{thm:lowdissba}
	Whenever $0 < s < (\gamma+1)/2$, the equation \eqref{eq:bagrushin} is never exactly null-controllable from any
	control support $\omega \subset \RR^d \times \TT^d$ satisfying $\overline{\omega} \cap \{x=0\} = \emptyset$.
\end{thm}

Let us finish this section with an example.

\begin{ex}\label{ex:example}
	For some fixed length $L>0$, we consider the control support
	\[
		\omega_L
		=
		B(0,L)^c \times \TT^d \subset \RR^d \times \TT^d,
	\]
	and the associated control time
	\[
		T_{\gamma,s,L}
		=
		\inf\big\{ T>0 \colon \text{\eqref{eq:bagrushin} is exactly null-controllable from $\omega_L$ at time $T$} \big\}
		.
	\]
	It is easy to see that for all $\varepsilon > 0$ the set $B(0,L)^c$ is $(\gamma_{\varepsilon},L_{\varepsilon})$-thick in
	$\RR^d$ with
	\[
		L_{\varepsilon}
		=
		L + \varepsilon
		\quad\text{ and }\quad
		\gamma_{\varepsilon}
		=
		1 - \frac{L^d}{(L+\varepsilon)^d}
		.
	\]
	Since the control support $\omega_L$ also satisfies the geometric condition $\overline{\omega}_L \cap \{x=0\} = \emptyset$, it
	follows from Theorems~\ref{thm:bagrushincont},~\ref{thm:critdiss}, and~\ref{thm:lowdissba} that
	\begin{equation}\label{eq:threeregimes}
		\begin{cases}
			T_{\gamma,s,L} = 0 & \text{when } s > (1+\gamma)/2, \\[5pt]
			0<T_{\gamma,s,L} < +\infty & \text{when } s = (1+\gamma)/2, \\[5pt]
			T_{\gamma,s,L} = +\infty & \text{when } s < (1+\gamma)/2.
		\end{cases}
	\end{equation}
	In the critical dissipation regime $s = (1+\gamma)/2$, we actually have from Theorem~\ref{thm:critdiss} for all
	$\varepsilon > 0$ the more precise two-sided estimate
	\begin{equation}\label{eq:inminitime}
		\frac1{1+\gamma}\bigg(\frac L{\sqrt{\lambda_{\gamma}}}\bigg)^{1+\gamma}
		\leq
		T_{\gamma,(1+\gamma)/2,L}
		\leq
		c_{\gamma}\,\biggl( \frac{L+\varepsilon}{\sqrt{\lambda_{\gamma}}} \biggr)^{1+\gamma}\,
			\log\biggl( \frac{c_{\gamma}(L+\varepsilon)^d}{(L+\varepsilon)^d-L^d} \biggr)
		,
	\end{equation}
	where $\lambda_{\gamma} > 0$ denotes again the smallest eigenvalue of the anharmonic oscillator $H_{\gamma,1}$. Moreover,
	as stated in Corollary~\ref{cor:asymptotics} below, $\lambda_{\gamma}$ converges to $\lambda_D$ as $\gamma$ goes to
	$+\infty$, where $\lambda_D>0$ stands for the smallest eigenvalue of the Dirichlet Laplacian on the canonical Euclidean unit
	ball $B(0,1)$ in $\RR^d$ (this is a quite straightforward consequence of the theory of large coupling limit).
	Since then
	\[
		\frac{L}{\sqrt{\lambda_\gamma}}
		\to
		\frac{L}{\sqrt{\lambda_D}}
		,
	\]
	we immediately infer that
	\begin{equation}\label{eq:minitime}
		\frac1{1+\gamma}\bigg(\frac L{\sqrt{\lambda_{\gamma}}}\bigg)^{1+\gamma}
		\to
		\begin{cases}
			+\infty & \text{when } L > \sqrt{\lambda_D},\\[5pt]
			0 & \text{when } L < \sqrt{\lambda_D}.
		\end{cases}
	\end{equation}
	Together with \eqref{eq:inminitime}, the latter implies, in particular, that, as $\gamma\rightarrow+\infty$,
	\[
		T_{\gamma,(1+\gamma)/2,L}\rightarrow +\infty\quad \text{when $L>\sqrt{\lambda_D}$}.
	\]
	Moreover, further calculations suggest that the first instance of the constant $c_\gamma$ in \eqref{eq:inminitime} can be
	replaced by $c^\gamma$ with some constant $c > 1$ that does not depend on the dimension, and that the second instance can be
	replaced by a constant not depending on $\gamma$. As a consequence, we have $T_{\gamma,(1+\gamma)/2,L} \to 0$ as
	$\gamma \to +\infty$ for $L < \sqrt{\lambda_D}/c$; the regime $\sqrt{\lambda_D}/c \leq L \leq \sqrt{\lambda_D}$ is still unclear
	at the moment. In any case, since $\sqrt{\lambda_D}$ approaches $+\infty$ as the dimension $d$ goes to $+\infty$, the asymptotic
	behaviour of $T_{\gamma,(1+\gamma)/2,L}$ depending on $L$ in this fashion, and not, as one might expect a priori, on $L > 1$ and
	$L < 1$, respectively, is quite surprising. Moreover, as mentioned in Remark~\ref{rk:cridiss}\,$(ii)$, the quantity
	\eqref{eq:minitime} is consistent with minimal times appearing in the study of Grushin-type models. Motivated by this, we
	conjecture that the lower bound in \eqref{eq:inminitime} is actually an equality, that is,
	\[
		T_{\gamma,(1+\gamma)/2,L}
		=
		\frac1{1+\gamma}\biggl( \frac L{\sqrt{\lambda_{\gamma}}} \biggr)^{1+\gamma}
		.
	\]
	The relevant regimes of $L$ for the asymptotic behaviour of $T_{\gamma,(1+\gamma)/2,L}$ as $\gamma \to +\infty$ would then be
	$L > \sqrt{\lambda_D}$ and $L < \sqrt{\lambda_D}$, that is,
	\[
		T_{\gamma,(1+\gamma)/2,L}
		\to
		\begin{cases}
			+\infty & \text{when } L > \sqrt{\lambda_D},\\[5pt]
			0 & \text{when } L < \sqrt{\lambda_D}.
		\end{cases}
	\]
\end{ex}

\begin{rk}
	Incidentally, as explained in Remark~\ref{rk:schrogrushin} below, the analogous proof as the one for
	Theorem~\ref{thm:critdiss}\,$(ii)$ yields that the fractional Schr\"odinger-Baouendi-Grushin equation
	\begin{equation}\label{eq:sbagrushin}\tag{$SE_{\gamma,s}$}
		\begin{cases}
			i\partial_tf(t,x,y) + (- \Delta_{\gamma})^sf(t,x,y) = h(t,x,y)\bmone_{\omega}(x,y),\ t\in\RR, (x,y)\in\RR^d\times\TT^d, \\[5pt]
			f(0,\cdot,\cdot) = f_0\in L^2(\RR^d\times\mathbb T^d),
		\end{cases}
	\end{equation}
	which is the oscillatory counterpart of the equation \eqref{eq:bagrushin}, is never exactly null-controllable from any control
	support $\omega \subset \RR^d \times \TT^d$ satisfying the condition $\overline{\omega} \cap \{x=0\} = \emptyset$. This
	difference in behavior between the equations \eqref{eq:bagrushin} and \eqref{eq:sbagrushin} contrasts with what is known for the
	heat and the corresponding Schr\"odinger equation, see e.g. \cite[Section 2.2]{MPS21}.
\end{rk}

\begin{rk}\label{rk:refgrushin}
	The results presented in this subsection are in line with articles devoted to the study of the
	null-controllability of Grushin-type heat equations. A pioneering article in this theory is \cite{BCG}, which paved
	the way for a numerous series of articles of which we can cite \cite{ABM, BDE, BMM, DKR, DK, Ko}. All these works illustrate the
	fact that the null-controllability of Grushin-type heat equations is governed by minimal times as in
	Example~\ref{ex:example}, and some of these works are even devoted to the computation of these times. Let us also mention
	that the null-controllability of the Schrödinger-Grushin equation is studied in the papers \cite{BS,LS}.
\end{rk}

\section{Spectral inequalities for the anisotropic Shubin operators}\label{sec:spectralEst}

The objective of this section is to prove Theorems~\ref{thm:mainharmo} and~\ref{thm:main}. To this end, we mainly focus on
proving the latter result and then explain briefly how its proof can be adapted in order to obtain the stronger spectral
inequality for the harmonic oscillator in Theorem~\ref{thm:mainharmo}.

\subsection{An abstract uncertainty relation}\label{sec:abstract}
Let us begin with recalling from \cite{EgidiS-21} the abstract result that plays an essential role in obtaining our spectral
inequalities. In order to give its statement, we need to introduce the following definition: given a domain
$\Omega \subset \RR^d$, a constant $\kappa \geq 1$, and a length $l > 0$, we call a finite or countably infinite family
$\{ Q_j \}_j$ of non-empty bounded convex open subsets $Q_j \subset \Omega$ a \emph{$(\kappa,l)$-covering of $\Omega$} if
\begin{enumerate}
    \item[$(i)$]
    the set $\Omega \setminus \bigcup_j Q_j$ has Lebesgue measure zero;

    \item[$(ii)$]
    each $Q_j$ is contained in a hypercube with sides of length $l$ parallel to coordinate axes;

    \item[$(iii)$]
    the estimate $\sum_j \norm{ g }_{L^2(Q_j)}^2 \leq \kappa \norm{ g }_{L^2(\Omega)}^2$ holds for all $g \in L^2(\Omega)$.
\end{enumerate}

We now have the following particular case of an uncertainty relation from \cite{EgidiS-21}.

\begin{prop}[{\cite[Proposition~3.1]{EgidiS-21}}]\label{prop:ES}
	Let $\{ Q_j \}_j$ be a $(\kappa,l)$-covering of a given domain $\Omega \subset \RR^d$, and suppose that
	$f \in \bigcap_{n\in\NN} W^{n,2}(\Omega)$ satisfies
	\[
		\forall n\in\mathbb N,\quad\sum_{\abs{\alpha}=n} \frac{1}{\alpha!} \norm{ \partial_x^\alpha f }_{L^2(\Omega)}^2
		\leq
		\frac{C_B(n)}{n!} \norm{ f }_{L^2(\Omega)}^2
		,
	\]
	with constants $C_B(n) > 0$ such that
	\[
		h
		:=
		\sum_{n\in\NN} \sqrt{C_B(n)}\,\frac{(10dl)^n}{n!}
		<
		\infty
		.
	\]
	Then, for every measurable subset $\omega \subset \Omega$ satisfying
	$\tau := \inf_j \abs{ Q_j \cap \omega } / \mathrm{diam}(Q_j)^d > 0$, we have
	\[
		\norm{ f }_{L^2(\Omega)}^2
		\leq
		\frac{\kappa}{6} \bigg( \frac{24d\abs{B(0,1)}}{\tau} \bigg)^{2\frac{\log\kappa}{\log2}+4\frac{\log h}{\log2}+5}
			\norm{ f }_{L^2(\omega)}^2
		.
	\]
\end{prop}

In view of Proposition \ref{prop:ES}, we therefore need in the following to prove so-called Bernstein inequalities of the form
\begin{equation}\label{eq:Bernstein}
	\sum_{\vert\alpha\vert=n} \frac1{\alpha!} \norm{\partial_x^\alpha f}_{L^2(\Omega)}^2
	\leq
	\frac{C_B(n,\lambda)}{n!} \norm{f}_{L^2(\Omega)}^2
	,\quad
	n\in\mathbb N,\,f\in\cE_{\lambda,k,m}
	,
\end{equation}
with a properly chosen domain $\Omega \subset \RR^d$. In order to alleviate the writing, we use throughout this section the
abbreviations
\[
	\mu
	:=
	\frac{k}{k+m}
	,\quad
	\nu
	:=
	\frac{m}{k+m}
	,\quad
	\scrit
	:=
	\frac1{2k}+\frac1{2m}
	.
\]

\subsection{Agmon estimates for spectral subspaces}\label{ssec:Agmon}
A key ingredient in obtaining inequalities of the form \eqref{eq:Bernstein} is given by the following variant of Agmon estimates
from \cite{Alp20b} for spectral subspaces associated with the (anisotropic) Shubin operators $H_{k,m}$.

\begin{prop}\label{prop:Agmon}
	There exist positive constants $c_1,c_2,c_3>0$ and $t_0 \in (0,1]$, depending only on $k$, $m$, and the dimension $d$, such that
	for all $t \in [0,t_0)$, $\lambda \geq 0$, and $f \in \cE_{\lambda,k,m}$ we have
	\begin{equation}\label{eq:agmon}
		\norm{ \euler^{c_1t\sprod{x}^{1/\nu}}f }^2_{L^2(\RR^d)} + \norm{ \euler^{c_1t\sprod{D_x}^{1/\mu}}f }^2_{L^2(\RR^d)}
		\leq
		c_2\lambda^{d\scrit}\euler^{c_3t\lambda^{\scrit}} \norm{f}^2_{L^2(\RR^d)}
		.
	\end{equation}
\end{prop}

\begin{proof}
	We know from \cite[Theorem 2.1]{Alp20b} that there exist some positive constants $c_1,\tilde{c} > 0$, and $t_0 \in (0,1]$ such
	that for every normalized eigenfunction $\psi \in L^2(\RR^d)$ of the operator $H_{k,m}$ and all $t\in[0,t_0)$ we
	have
	\[
		\norm{ \euler^{c_1t\sprod{x}^{1/\nu}}\psi }_{L^2(\RR^d)} + \norm{ \euler^{c_1t\sprod{D_x}^{1/\mu}}\psi }_{L^2(\RR^d)}
		\leq
		\tilde{c}\euler^{\tilde{c}t\lambda^\scrit}
		,
	\]
	where $\lambda > 0$ is the eigenvalue associated with the eigenfunction $\psi$; recall that $H_{k,m}$ has purely discrete
	spectrum. Expanding $f \in \cE_{\lambda,k,m}$ for $\lambda \geq 0$ as a linear combination of eigenfunctions, we therefore
	deduce that for all $t\in[0,t_0)$ we have
	\[
		\norm{ \euler^{c_1t\sprod{x}^{1/\nu}}f }_{L^2(\RR^d)}^2 + \norm{ \euler^{c_1t\sprod{D_x}^{1/\mu}}f }_{L^2(\RR^d)}^2
		\leq
		N(\lambda)\tilde{c}^2\euler^{2\tilde{c}t\lambda^\scrit}\norm{f}_{L^2(\RR^d)}^2
		,
	\]
	where $N(\lambda)$ is chosen as the number of distinct eigenvalues of $H_{k,m}$ less or equal to $\lambda$. Using the Weyl
	law asymptotics from \cite[Remark~5.7]{ChatzakouDR-21} for the eigenvalue counting function associated to $H_{k,m}$, cf.\ also
	\cite[Theorem~2.3.2]{BoggiattoBR-96}, we then observe that
	\[
		N(\lambda)
		\leq
		c'\lambda^{d\scrit},
	\]
	with some constant $c' > 0$ depending only on $k$, $m$, and $d$. The proof is then ended upon choosing $c_2 = \tilde{c}^2c'$ and
	$c_3 = 2\tilde{c}$.
\end{proof}

\begin{rk}\label{rk:agmon}
	(1)
	The term $\lambda^{d\scrit}$ on the right-hand side of \eqref{eq:agmon} is unexpected, and we indeed conjecture that
	\eqref{eq:agmon} holds without this term, that is,
	\begin{equation}\label{eq:sharpagmon}
		\norm{ \euler^{c_1t\sprod{x}^{1/\nu}}f }^2_{L^2(\RR^d)} + \norm{ \euler^{c_1t\sprod{D_x}^{1/\mu}}f }^2_{L^2(\RR^d)}
		\leq
		c_2\euler^{c_3t\lambda^{\scrit}} \norm{f}^2_{L^2(\RR^d)}
		.		
	\end{equation}	
	The reason the term $\lambda^{d\scrit}$ appears in \eqref{eq:agmon} lies in the way we carry quantitative sharp Agmon
	estimates for single eigenfunctions of the operator $H_{m,k}$ over to finite linear combinations of eigenfunctions. To the best
	of our knowledge, there are very few results in the literature stating Agmon estimates for spectral subspaces which are sharp
	with respect to possible parameters involved ($t\in [0,t_0)$ in this case for us), the rare exception being the case of the
	harmonic oscillator, see \cite[Proposition~3.3]{BJKPS}. Proving the stronger estimates \eqref{eq:sharpagmon} would immediately
	allow us to remove the unfavorable term $\log(1+\lambda)$ in the spectral inequalities \eqref{eq:specIneq}.

	(2)
	In the particular case of $m = 1$, one may take $c_1 = \nu = 1 / (k+1)$ and $t_0 = 1$ in Proposition~\ref{prop:Agmon}. This
	follows from the above reasoning by simply replacing the Agmon estimates for single eigenfunctions from \cite[Theorem~2.1]{Alp20b}
	by more explicit ones for $m = 1$ with the mentioned values of $c_1$ and $t_0$, which can be obtained, for instance,
	by suitably adapting the proof in [2]. These more precise Agmon estimates are also consistent with classical ones from the
	literature, see, e.g., \cite[Theorem 3.4]{Hislop}.

\end{rk}

\subsection{Bernstein inequalities}\label{ssec:Bernstein}
Proposition~\ref{prop:Agmon} now allows us to prove a global Bernstein inequality, that is, an inequality of the form
\eqref{eq:Bernstein} with $\Omega = \mathbb R^d$.

\begin{prop}\label{prop:globalBernstein}
	There exist positive constants $c,C>0$, depending only on $k$, $m$, and the dimension $d$, such that for all $n\geq0$,
	$\delta>0$, $\lambda \geq 0$ and $f \in \cE_{\lambda,k,m}$ we have
	\[
		\sum_{\abs{\alpha} = n} \frac1{\alpha!} \norm{ \partial_x^\alpha f }_{L^2(\RR^d)}^2
		\leq
		\frac{C_B(n,\lambda,\delta)/2}{n!}\norm{f}_{L^2(\RR^d)}^2
	\]
	with
	\begin{equation}\label{eq:BernsteinConst}
		C_B(n,\lambda,\delta)
		=
		2C^{2(1+n)}\,\delta^{2n}\,(n!)^2 (1+\lambda^{d\scrit}) \euler^{(c+d)\delta^{-1/\nu}} \euler^{c\delta^{-1}\lambda^{\frac1{2m}}}.
	\end{equation}
\end{prop}

\begin{proof}
	Using integration by parts (see also Lemma~2.1 and Remark~2.2 in \cite{EgidiS-21}) and Plancherel's theorem, we have
	\[
		\sum_{\abs{\alpha}=n} \frac1{\alpha!} \norm{ \partial_x^\alpha f }_{L^2(\RR^d)}^2
		=
		\frac1{n!} \sprod{ (-\Delta)^nf , f }_{L^2(\RR^d)}
		=
		\frac1{n!} \sprod{ \abs{\xi}^{2n} \hat{f} , \hat{f} }_{L^2(\RR^d)}
		=
		\frac1{n!} \norm{ \abs{\xi}^n \hat{f} }_{L^2(\RR^d)}^2
		,
	\]
	where $\hat f$ denotes the Fourier transform of the function $f$. We therefore have to estimate the quantity
	$\norm{ \abs{\xi}^n \hat{f} }_{L^2(\RR^d)}$. Note here that $\hat{f}$ belongs to $\cE_{\lambda,m,k}$ since $H_{k,m}$ is similar
	to $H_{m,k}$ by Fourier transform.

	With $c_1,c_2,c_3 > 0$ and $t_0 \in (0,1]$ as in Proposition~\ref{prop:Agmon} and $t \in (0,t_0)$, we write
	\[
		\abs{\xi}^n
		=
		\abs{\xi}^n \euler^{-c_1t\sprod{\xi}^{1/\mu}} \euler^{c_1t\sprod{\xi}^{1/\mu}},
	\]
	and estimate
	\[
		\norm{ \abs{\xi}^n\hat{f} }_{L^2(\RR^d)}
		\leq
		\norm{ \sprod{\xi}^n \euler^{-c_1t\sprod{\xi}^{1/\mu}} }_{L^\infty(\RR^d)}
			\norm{ \euler^{c_1t\sprod{\xi}^{1/\mu}}\hat f }_{L^2(\RR^d)},
	\]
	with, moreover,
	\[
		\norm{ \sprod{\xi}^n \euler^{-c_1t\sprod{\xi}^{1/\mu}} }_{L^\infty(\RR^d)}
		=
		\sup_{r\geq1}r^n \euler^{-c_1tr^{1/\mu}}
		\leq
		\Bigl( \frac{\mu}{c_1\euler t} \Bigr)^{n\mu} n^{n\mu}
		\leq
		\Bigl( \frac{\mu}{c_1t} \Bigr)^{n\mu} (n!)^\mu
		.
	\]
	Applying Proposition~\ref{prop:Agmon} to $\hat{f} \in \cE_{\lambda,m,k}$ and taking into account that
	$\norm{\hat{f}}_{L^2(\RR^d)} = \norm{f}_{L^2(\RR^d)}$, we thus obtain from the above that
	\[
		\sum_{\abs{\alpha} = n} \frac1{\alpha!} \norm{ \partial_x^\alpha f }_{L^2(\RR^d)}^2 
		\leq
		\frac{c_2}{n!}\, \Bigl(\frac{\mu}{c_1t}\Bigr)^{2n\mu}\,(n!)^{2\mu}\,\lambda^{d\scrit} \euler^{c_3t\lambda^\scrit}
			\norm{f}_{L^2(\RR^d)}^2
		.
	\]

	Suppose that $\lambda > (1/\delta)^{2k}$. With the particular choice
	$t = t_0\mu\delta^{-1}\lambda^{-1/(2k)} < t_0\mu \leq t_0$, we then have
	\[
		\sum_{\abs{\alpha} = n} \frac1{\alpha!} \norm{ \partial_x^\alpha f }_{L^2(\RR^d)}^2
		\leq
		\frac{c_2}{n!}\,\Bigl(\frac{1}{c_1t_0}\Bigr)^{2n\mu}\,\delta^{2n\mu}\,\lambda^{\frac{n\mu}k}\,(n!)^{2\mu}\,\lambda^{d\scrit}
			\euler^{c_3\mu\delta^{-1}\lambda^{\frac1{2m}}} \norm{f}_{L^2(\RR^d)}^2
		.
	\]
	We further estimate
	\[
		(\lambda^{\frac1{2k}})^{\mu n}
		=
		(\lambda^{\frac1{2m}})^{\nu n}
		=
		\delta^{\nu n} \Bigl(\frac{\lambda^{\frac1{2m}}}{\delta} \Bigr)^{\nu n}
		\leq
		\delta^{\nu n} (n!)^{\nu} \euler^{\nu\delta^{-1}\lambda^{\frac1{2m}}}
		.
	\]
	Combining the last two inequalities, and taking into account that $\mu + \nu = 1$, we conclude that 
	\begin{equation}\label{eq:globalBernsteinLarge}
		\sum_{\abs{\alpha} = n} \frac1{\alpha!} \norm{ \partial_x^\alpha f }_{L^2(\RR^d)}^2
		\leq
		\frac{C^{2(1+n)}}{n!}\,\delta^{2n}\,(n!)^2\, \lambda^{d\scrit} \euler^{c\delta^{-1}\lambda^{\frac1{2m}}}\norm{f}_{L^2(\RR^d)}^2
	\end{equation}
	with $C^2 = \max\{ c_2 , (c_1t_0)^{-\mu} \}$ and $c = \max\{ 2 , c_3 \}$.

	It remains to consider the case $\lambda \leq (1/\delta)^{2k}$. Since then $\cE_{\lambda,k,m} \subset \cE_{(1/\delta)^{2k},k,m}$,
	we obtain from \eqref{eq:globalBernsteinLarge} with $\lambda$ replaced by $(1/\delta)^{2k}$ that
	\begin{equation}\label{eq:globalBernsteinSmall}
		\begin{aligned}
			\sum_{\abs{\alpha} = n} \frac1{\alpha!} \norm{ \partial_x^\alpha f }_{L^2(\RR^d)}^2
			&\leq
			\frac{C^{2(1+n)}}{n!}\,\delta^{2n}\,(n!)^2\, \delta^{-2kd\scrit} \euler^{c\delta^{-1}\delta^{-\frac km}}
				\norm{f}_{L^2(\RR^d)}^2\\
			&\leq
			\frac{C^{2(1+n)}}{n!}\,\delta^{2n}\,(n!)^2\, \euler^{(c+d)\delta^{-1/\nu}} \norm{f}_{L^2(\RR^d)}^2
			,
		\end{aligned}
	\end{equation}
	where for the last inequality we used $\delta^{-2kd\scrit} = \delta^{-d/\nu} \leq \euler^{d\delta^{-1/\nu}}$. In light of
	$\euler^{(c+d)\delta^{-1/\nu}} \geq 1$ and $\euler^{\delta^{-1}\lambda^{\frac1{2m}}} \geq 1$, the claim now follows from
	\eqref{eq:globalBernsteinLarge} and \eqref{eq:globalBernsteinSmall}.
\end{proof}

\begin{rk}\label{rk:Bernstein}
	(1)
	Bernstein inequalities closely related to Proposition~\ref{prop:globalBernstein} have recently been obtained in
	\cite[Eq.\ (4.5)]{Martin} using smoothing properties of the semigroup associated to (fractional powers of) $H_{k,m}$ established
	in \cite{Alp20b}. These smoothing properties also rely on the Agmon estimates for eigenfunctions, so that our proof above is
	more direct. Moreover, our constant in \eqref{eq:BernsteinConst} incorporates the parameter $\delta$, which may be used to force
	convergence of an associated series, see \eqref{eq:defh} below, and thus makes our inequality more suitable for our purposes.

	(2)
	In the particular case of the harmonic oscillator, that is, $k = m = 1$, Bernstein inequalities without the unfavorable term
	$1+\lambda^{d\scrit}$ have already been obtained in the literature. More precisely, \cite[Proposition~B.1]{EgidiS-21} (cf.\ also
	\cite[Proposition~3.3\,(i)]{BJKPS}) states that for all $n \geq 0$, $\delta > 0$, $\lambda \geq 0$, and
	$f \in \cE_{\lambda,1,1}$ we have
	\[
		\sum_{\abs{\alpha} = n} \frac1{\alpha!} \norm{ \partial_x^\alpha f }_{L^2(\RR^d)}^2
		\leq
		\frac{C_B(n,\lambda,\delta)/2}{n!}\norm{f}_{L^2(\RR^d)}^2
	\]
	with
	\begin{equation}\label{eq:BernsteinHarmOsc}
		C_B(n,\lambda,\delta)
		=
		2(2\delta)^{2n}\,(n!)^2\,\euler^{e\delta^{-2}} \euler^{2\delta^{-1}\sqrt{\lambda}}
		.
	\end{equation}
\end{rk}

We are finally able to derive the local Bernstein inequalities of the desired form.

\begin{cor}\label{cor:locBernstein}
	Let $\lambda > 0$, and let $\Omega \subset \RR^d$ be an open set containing the ball $B(0,(2\lambda)^{1/2k})$. Then, for all
	$n\geq0$, $\delta>0$, and $f \in \cE_{\lambda,k,m}$ we have
	\[
		\norm{f}_{L^2(\RR^d)}^2
		\leq
		2\norm{f}_{L^2(\Omega)}^2
		,
	\]
	and
	\[
		\sum_{\abs{\alpha} = n}\frac1{\alpha!} \norm{ \partial_x^\alpha f }_{L^2(\Omega)}^2
		\leq
		\frac{C_B(n,\lambda,\delta)}{n!} \norm{f}_{L^2(\Omega)}^2
		,
	\]
	with $C_B(n,\lambda,\delta)$ as in \eqref{eq:BernsteinConst}.
\end{cor}

\begin{proof}
	We have
	\[
		\norm{ \abs{x}^k f }_{L^2(\RR^d)}^2
		=
		\sprod{ \abs{x}^{2k}f , f }_{L^2(\RR^d)}
		\leq
		\sprod{ H_{k,m}f , f }_{L^2(\RR^d)}
		\leq
		\lambda\norm{f}_{L^2(\RR^d)}^2
		,
	\]
	where the last inequality follows by functional calculus. Hence,
	\[
		\norm{f}_{L^2(\RR^d\setminus B(0,(2\lambda)^{1/2k}))}^2
		=
		\norm{ \abs{x}^{-k} \abs{x}^k f }_{L^2(\RR^d\setminus B(0,(2\lambda)^{1/2k}))}^2
		\leq
		\frac1{2\lambda} \norm{ \abs{x}^k f }_{L^2(\RR^d)}^2
		\leq
		\frac12\norm{f}_{L^2(\RR^d)}^2
		.
	\]
	Since $\Omega$ contains the ball $B(0,(2\lambda)^{1/2k})$ by hypothesis, this implies that
	\[
		\norm{f}_{L^2(\RR^d)}^2
		\leq
		2\norm{f}_{L^2(B(0,(2\lambda)^{1/2k}))}^2
		\leq
		2\norm{f}_{L^2(\Omega)}^2
		.
	\]
	Moreover, we deduce from Proposition~\ref{prop:globalBernstein} that
	\[
		\sum_{\abs{\alpha}=n} \frac{1}{\alpha!} \norm{ \partial_x^\alpha f }_{L^2(\Omega)}^2
		\leq
		\sum_{\abs{\alpha}=n} \frac{1}{\alpha!} \norm{ \partial_x^\alpha f }_{L^2(\RR^d)}^2
		\leq
		\frac{C_B(n,\lambda,\delta)/2}{n!} \norm{f}_{L^2(\RR^d)}^2
		,
	\]
	which, together with the former inequality, proves the claim.
\end{proof}

For future reference and in light of Proposition~\ref{prop:ES}, we now consider for $\delta,l > 0$ and $\lambda > 0$ the quantity
\begin{align*}
	h(l,\lambda,\delta)
	&:=
	\sum_{n\geq0} \sqrt{C_B(n,\lambda,\delta)}\, \frac{(10dl)^n}{n!} \\
	&=
	\sqrt 2C\, (1+\lambda^{d\scrit})^{1/2} \euler^{(c+d)2^{-1}\delta^{-1/\nu}} \euler^{c(2\delta)^{-1}\lambda^{\frac1{2m}}}
		\sum_{n\geq0}(10dlC\delta)^n
	.
\end{align*}
With the particular choice $\delta^{-1} = 20dlC$, we deduce that there is a constant $C' > 0$, depending only on $k$, $m$,
and the dimension $d$, such that
\begin{equation}\label{eq:defh}
	h(l,\lambda)
	:=
	h(l,\lambda,(20dlC)^{-1})
	\leq
	C' (1+\lambda^{d\scrit})^{1/2} \euler^{C'l^{1/\nu}} \euler^{C'l\lambda^{\frac{1}{2m}}}
	.
\end{equation}

\subsection{Conclusion of Theorem~\ref{thm:main}}
Let $\omega \subset \RR^d$ be a measurable set as in \eqref{eq:omega}, and let $f \in \cE_{\lambda,k,m}$ with
$\lambda \geq 0$. Consider $\lambda_{k,m} := \min\spec(H_{k,m}) > 0$. Then, if $\lambda \in [0,\lambda_{k,m})$, we have
$\cE_{\lambda,k,m} = \{0\}$ and there is nothing to prove. It therefore suffices to consider $\lambda \geq \lambda_{k,m} > 0$.

The key step is to use the well-known Besicovitch covering theorem in the following formulation taken from
\cite[Proposition~7.1]{DickeSV-23}; see also \cite[Theorem~2.7]{Mattila-99}.

\begin{prop}[Besicovitch]
	Let $A \subset \RR^d$ be bounded, and let $\cB$ be a family of open balls such that each point in $A$ is the center of some ball
	from $\cB$. Then there are at most countably many balls $(B_j)_j \subset \cB$ such that
	\begin{equation}\label{eq:besi}
		\mathbbm1_A
		\leq
		\sum_j\mathbbm1_{\overline{B}_j}
		\leq
		K_{\mathrm{Bes}}^d
		,
	\end{equation}
	where $K_{\mathrm{Bes}} \geq 1$ is a universal constant.
\end{prop}

We are finally in position to prove Theorem~\ref{thm:main}.

\begin{proof}[Proof of Theorem~\ref{thm:main}]
	Suppose that $\lambda \geq \lambda_{k,m} > 0$, and let
	\[
		A
		:=
		B(0,(2\lambda)^{1/2k})
		\quad\text{ and }\quad
		\cB
		:=
		\bigl\{B(x,\rho(x)) \colon x \in A \bigr\}
		.
	\]
	Besicovitch's covering theorem then implies that there is a finite or countably infinite collection of points $x_j \in A$
	such that \eqref{eq:besi} holds with $B_j = B(x_j,\rho(x_j))$. In particular, $A$ is contained in the union
	$\bigcup_j \overline{B}_j$. Let $\Omega$ be the interior of $\bigcup_j \overline{B}_j$. Then, $\Omega$ is open and contains the
	open set $A$ by definition. Moreover, it is easy to see that $\Omega$ is a domain.

	With $\theta_{\lambda,k}$ and $L_{\lambda,k}$ from \eqref{eq:geom}, for each $j$ we clearly have
	$\sigma(x_j) \geq \theta_{\lambda,k}$ and $\rho(x_j) \leq L_{\lambda,k}$. Hence, the family $\{B_j\}_j$ gives a
	$(K_{\mathrm{Bes}}^d , L_{\lambda,k})$-covering of $\Omega$ in the sense of Section~\ref{sec:abstract}, and from
	\eqref{eq:omega} we have
	\[
		\inf_j \frac{\abs{ \omega \cap B_j }}{\diam(B_j)^d}
		=
		\frac{\abs{B(0,1)}}{2^d} \inf_j \frac{\abs{ \omega\cap B(x_j,\rho(x_j)) }}{\abs{B(x_j,\rho(x_j))}}
		\geq
		\frac{\abs{B(0,1)}}{2^d}\,\inf_j \sigma(x_j)
		\geq
		\frac{\abs{B(0,1)}}{2^d}\,\theta_{\lambda,k}
		.
	\]
	Taking into account Corollary~\ref{cor:locBernstein} and \eqref{eq:defh}, applying Proposition~\ref{prop:ES} with
	$\{ Q_j \}_j = \{ B_j \}_j$, $l = L_{\lambda,k}$, and $h(\lambda) = h(L_{\lambda,k},\lambda)$ therefore yields
	\[
		\norm{f}_{L^2(\RR^d)}^2
		\leq
		2\norm{f}_{L^2(\Omega)}^2
		\leq
		\frac{K_{\mathrm{Bes}}^d} 3\biggl(\frac{24d2^d}{\theta_{\lambda,k}}\biggr)^{2\frac{\log K_{\mathrm{Bes}}^d}{\log2}
			+ 4\frac{\log h(\lambda)}{\log 2} + 5} \norm{f}_{L^2(\omega\cap\Omega)}^2
		.
	\]
	Here, we observe that for all $r \geq 0$, we have
	\[
		1+\lambda^r
		\leq
		(1+\lambda_{k,m}^{-r}) \lambda^r
		\leq
		(1+\lambda_{k,m}^{-r}) (1+\lambda)^r
		,
	\]
	so that
	\[
		\log h(\lambda)
		\leq
		\frac{1}{2}\log\bigl( (C')^2(1+\lambda_{k,m}^{-d\scrit}) \bigr) + \frac{d\scrit}{2}\log(1+\lambda)
			+ C'(L_{\lambda,k})^{1/\nu}
			+ C'L_{\lambda,k}\lambda^{\frac1{2m}}
		.
	\]
	The claim therefore follows from the above upon an appropriate choice of the constant $K$, depending on $d$, $C'$,
	$\lambda_{k,m}$, $\nu$, $\scrit$, and $K_{\mathrm{Bes}}$, that is, effectively only on $d$, $k$, and $m$.
\end{proof}

We close this section by briefly discussing how the proof of Theorem~\ref{thm:main} can be adapted to obtain
Theorem~\ref{thm:mainharmo}.

\begin{proof}[Proof of Theorem~\ref{thm:mainharmo}]
	As mentioned in Remark~\ref{rk:Bernstein}\,(2), in the particular case of the harmonic oscillator, that is, $k = m = 1$, there
	are Bernstein inequalities available that do not contain the unfavorable term $1 + \lambda^{d\scrit}$. Upon replacing the
	constant \eqref{eq:BernsteinConst} by \eqref{eq:BernsteinHarmOsc}, one can then follow the proof of Theorem~\ref{thm:main}
	verbatim towards a proof of Theorem~\ref{thm:mainharmo}, thereby avoiding the term $\log(1+\lambda)$ in the final estimate.
\end{proof}

\section{Proof of the exact null-controllability results}\label{sec:cont}
In this last main section we use the spectral inequalities given by Theorems~\ref{thm:mainharmo} and~\ref{thm:main}
in order to prove the exact null-controllability results from Section~\ref{ssec:exactContr} for the evolution equations
\eqref{eq:fractionalShubin} and \eqref{eq:bagrushin}.

Since the operators $H^s_{k,m}$ and $(-\Delta_{\gamma})^s$ are selfadjoint in $L^2(\mathbb R^d)$ and
$L^2(\RR^d\times\TT^d)$, respectively, the Hilbert Uniqueness Method implies that the exact null-controllability of these
equations is equivalent to the exact observability of the associated semigroups $(e^{-tH^s_{k,m}})_{t\geq0}$ and
$(e^{-t(-\Delta_{\gamma})^s})_{t\geq 0}$. The latter is defined as follows.

\begin{dfn}[Exact observability]
	Let $\tau>0$, and let $\Omega \subset \RR^d$ and $\omega \subset \Omega$ be measurable. A strongly continuous semigroup
	$(T(t))_{t\geq0}$ on $L^2(\Omega)$ is said to be exactly observable from the set $\omega$ in time $\tau$ if there exists a
	positive constant $C_{\omega,\tau} > 0$ such that for all $g\in L^2(\Omega)$, we have
	\[
		\norm{ T(\tau)g }^2_{L^2(\Omega)}
		\leq
		C_{\omega,\tau} \int_0^{\tau} \norm{ T(t)g }^2_{L^2(\omega)}\,\dd t
		.
	\]
\end{dfn}

In order to prove exact observability estimates, with an explicit observability constant $C_{\omega,\tau}$, we use the following
quantitative result that is based on the well-known Lebeau-Robbiano strategy and is particularly well adapted to the
equations we are studying.

\begin{thm}[{\cite[Theorem 2.8]{NakicTTV-20}}]\label{thm:obs}
	Let $A$ be a non-negative selfadjoint operator in $L^2(\RR^d)$, and let $\omega \subset \RR^d$ be measurable. Suppose that there
	are $d_0 > 0$, $d_1 \geq 0$, and $\power \in (0,1)$ such that for all $\lambda \geq 0$ and
	$f \in \cE_\lambda(A)$,
	\[
		\norm{f}^2_{L^2(\RR^d)}
		\leq
		d_0e^{d_1\lambda^{\power}} \norm{f}^2_{L^2(\omega)}
		.
	\]
	Then, there exist positive constants $c_1,c_2,c_3 > 0$, only depending on $\eta$, such that for all $T > 0$ and
	$g \in L^2(\RR^d)$ we have the observability estimate
	\[
		\norm{ e^{-TA}g }^2_{L^2(\RR^d)}
		\leq
		\frac{C_{\obs}}T\int_0^T \norm{ e^{-tA}g }^2_{L^2(\omega)} \,\dd t
		,
	\]
	where the positive constant $C_{\obs}>0$ is given by
	\[
		C_{\obs}
		=
		c_1d_0(2d_0+1)^{c_2}\exp\biggl( c_3\biggl( \frac{d_1}{T^{\power}} \biggr)^{\frac1{1-\power}} \biggr)
		.
	\]
\end{thm}

While the Lebeau-Robbiano strategy in Theorem~\ref{thm:obs} requires that the constant in the spectral inequality exhibits a
sublinear power growth in the exponent in terms of the spectral parameter $\lambda$, the following statement allows a more
general subexponential growth in $\lambda$, but does not provide a quantitative observability estimate.

\begin{thm}[{\cite[Theorem 5]{DM}}]\label{thm:millerduy}
	Let $A$ be a non-negative selfadjoint operator on $L^2(\RR^d)$, and let $\omega \subset\RR^d$ be measurable. Suppose that the
	spectral inequality
	\[
		\norm{f}^2_{L^2(\RR^d)}
		\leq
		ce^{c\lambda/((\log\log\lambda)^{\alpha}\log\lambda)} \norm{f}^2_{L^2(\omega)}
		,\quad
		f \in \cE_\lambda(A)
		,\,
		\lambda > e
		,
	\]
	holds with some $\alpha>2$ and $c>0$. Then, for all $T>0$, there exists a positive constant $C_T > 0$ such that for all
	$g \in L^2(\RR^d)$ we have
	\[
		\norm{ e^{-TA}g }^2_{L^2(\RR^d)}
		\leq
		C_T\int_0^T \norm{ e^{-tA}g }^2_{L^2(\omega)} \,\dd t
		.
	\]
\end{thm}

\subsection{Null-controllability of the Shubin evolution equations}
\label{ssec:proofCorollariesShubin}
Let us first focus on the results regarding the equation \eqref{eq:fractionalShubin}. Here, in order to deal with the fractional
powers of $H_{k,m}$, we use the fact that by the transformation formula for spectral measures, see, e.g.,
\cite[Proposition~4.24]{Schmuedgen-12}, for all $s>0$ and $\lambda \geq 0$ we have
\begin{equation}\label{eq:transfo}
	\cE_{\lambda,s,k,m}
	:=
	\bmone_{(-\infty,\lambda]}(H_{k,m}^s)
	=
	\bmone_{(-\infty,\lambda^{1/s}]}(H_{k,m})
	=
	\cE_{\lambda^{1/s},k,m}
	.
\end{equation}
In essence, this implies that a spectral inequality for $H_{k,m}$ yields a spectral inequality for $H_{k,m}^s$ by just replacing
$\lambda$ by $\lambda^{1/s}$ in the corresponding constant.

We are now in position to prove Corollaries~\ref{cor:fractionalShubin} and~\ref{cor:contlog}.

\begin{proof}[Proof of Corollary~\ref{cor:fractionalShubin}]
	Under the hypotheses on $\sigma$ and $\rho$, we are in the situation of Remark~\ref{rk:mainComparison} with $\delta \leq 1$.
	It therefore immediately follows from \eqref{eq:specweakly} and \eqref{eq:transfo} that for some constants $d_0 > 0$ and
	$d_1 \geq 0$ we have
	\[
		\norm{f}^2_{L^2(\RR^d)}
		\leq
		d_0 e^{d_1\lambda^{\power}} \norm{f}^2_{L^2(\omega)}
		,\quad
		f \in \cE_{\lambda,s,k,m} = \cE_{\lambda^{1/s},k,m}
		,
	\]
	with
	\[
		\power
		=
		\frac{\delta+a}{2sk} + \frac1{2sm}
		<
		1
		.
	\]
	The claim then immediately follows by applying Theorem~\ref{thm:obs}.
\end{proof}

\begin{proof}[Proof of Corollary~\ref{cor:contlog}]
	Given $s > a/2k + 1/2m$, we pick a $\delta\in(0,1)$ such that
	\[
		\frac{a}{2k}+\frac1{2m}
		<
		\frac{\delta+a}{2k}+\frac1{2m}
		<s
		.
	\]
	The hypothesis on $\rho$, namely $\rho(x) = o(\abs{x}^\delta)$ as $\abs{x} \to +\infty$, then implies that there is $L > 0$ such
	that
	\[
		\rho(x)
		\leq
		L\sprod{x}^\delta
		,\quad
		x \in \RR^d
		.
	\]
	We are thus in the situation of Corollary~\ref{cor:fractionalShubin} and the claim is just an instance of that result.
\end{proof}

While the two corollaries above rely on the Lebeau-Robbiano strategy from Theorem~\ref{thm:obs}, Corollary~\ref{cor:harmolog}
has to revert to the more general statement in Theorem~\ref{thm:millerduy}.

\begin{proof}[Proof of Corollary~\ref{cor:harmolog}]
	Under the hypothesis \eqref{eq:harmolog}, it is easy to see that for, say, $\lambda \geq e + 1$ we have
	\[
		L_\lambda
		=
		\sup_{\abs{x}<\sqrt{2\lambda}} \rho(x)
		\leq
		c'\,\frac{L\sqrt{\lambda}}{(\log\log\lambda)^\alpha \log\lambda}
	\]
	with a suitably chosen constant $c' > 0$ depending on $\alpha$, but not on $L$ or $\lambda$. It then follows from
	Theorem~\ref{thm:mainharmo} with a constant function $\sigma$ that
	\[
		\norm{f}^2_{L^2(\RR^d)}
		\leq
		ce^{c\lambda/((\log\log\lambda)^{\alpha}\log\lambda)} \norm{f}^2_{L^2(\omega)}
		,\quad
		f \in \cE_{\lambda,1,1}
		,\,
		\lambda \geq e + 1
		,
	\]
	where $c > 0$ is another constant, depending on $L$, $c'$, and the dimension $d$. Taking into account that
	$\cE_{\lambda,1,1} \subset \cE_{e+1,1,1}$ for $e < \lambda < e+1$, the latter even holds for all $\lambda > e$ after suitably
	adapting the constant $c$. The claim then immediately follows from Theorem~\ref{thm:millerduy}.
\end{proof}

\subsection{Null-controllability of the Baouendi-Grushin evolution equation}

Let us now turn to the null-controllability results for the degenerate parabolic equation \eqref{eq:bagrushin}. We
first observe that after passing to the Fourier side with respect to $\TT^d$-variable, the Baouendi-Grushin operator is
transformed as
\[
	\Delta_x + \abs{x}^{2\gamma}\Delta_y
	\,\rightsquigarrow\,
	\Delta_x - \abs{n}^2 \abs{x}^{2\gamma}
	,
\]
where $n \in \ZZ^d$ is the dual variable of $y \in \TT^d$. This motivates to
introduce the anharmonic oscillator $H_{\gamma;r}$ in $L^2(\RR^d)$ with variably scaled potential\footnote{This
notation is to be distinguished from the anisotropic Shubin operator $H_{k,m}$.},
\[
	H_{\gamma; r}
	:=
	-\Delta_x + r^2\abs{x}^{2\gamma}
	,\quad
	r \geq 0
	.
\]
Consequently, for all $g \in L^2(\RR^d\times\TT^d)$ and $(x,y) \in \RR^d \times \TT^d$ we have
\begin{equation}\label{eq:GrushinTransformed}
	(e^{-t(-\Delta_{\gamma})^s}g)(x,y)
	=
	\sum_{n\in\ZZ^d} e^{iy \cdot n} (e^{-tH_{\gamma;\abs{n}}^{s}} \hat{g}_n)(x)
	,
\end{equation}
where
\[
	\hat{g}_n
	=
	\int_{\TT^d} e^{-iy\cdot n} g(\cdot,y) \,\dd y
	.
\]

We first prove that the thickness condition is necessary to obtain a null-controllability result for the equation
\eqref{eq:bagrushin}.

\begin{proof}[Proof of Proposition \ref{prop:neccond}]
	Suppose that the equation \eqref{eq:bagrushin} is exactly null-controllable from a given measurable set
	$\omega \subset\RR^d \times \TT^d$ in some positive time $T>0$. This is equivalent to the existence of a positive
	constant $C_{\omega,T}>0$ such that for all	$g\in L^2(\RR^d\times\TT^d)$,
	\begin{equation}\label{eq:obsba}
		\norm{ e^{-T(-\Delta_{\gamma})^s}g }^2_{L^2(\RR^d\times\TT^d)}
		\leq
		C_{\omega,T} \int_0^T \norm{ e^{-t(-\Delta_{\gamma})^s}g }^2_{L^2(\omega)} \,\dd t
		.
	\end{equation}

	Now, every function $g \in L^2(\RR^d)$ can be treated as a function in $L^2(\RR^d \times \TT^d)$ that is constant with respect
	to the $\TT^d$-variable. As such, $\hat{g}_n$ in \eqref{eq:GrushinTransformed} then satisfies $\hat{g}_n = 0$ for $n \neq 0$ and
	$\hat{g}_0 = g$, so that from \eqref{eq:GrushinTransformed} we obtain for all $t \geq 0$,
	\[
		e^{-t(-\Delta_{\gamma})^s}g
		=
		e^{-t(-\Delta_x)^s}g
		.
	\]
	Inserting the latter into the observability estimate \eqref{eq:obsba}, we deduce that for all $g\in L^2(\RR^d)$,
	\begin{equation}\label{eq:obsbaII}
		\norm{ e^{-T(-\Delta_x)^s}g }^2_{L^2(\RR^d)}
		\leq
		C_{\omega,T} \int_0^T \norm{ e^{-t(-\Delta_x)^s}g }^2_{L^2(\omega)} \,\dd t
		.
	\end{equation}
	Moreover, by Fubini's theorem, the right-hand side of the latter inequality can for every $g \in L^2(\RR^d)$ be
	rewritten as
	\begin{equation}\label{eq:obsbaSlice}
		\int_0^T \norm{ e^{-t(-\Delta_x)^s}g }^2_{L^2(\omega)} \,\dd t
		=
		\int_{\TT^d} \int_0^T \norm{ e^{-t(-\Delta_x)^s}g }^2_{L^2(\omega_y)} \,\dd t \,\dd y
	\end{equation}
	with
	\[
		\omega_y
		=
		\{ x \in \RR^d \colon (x,y) \in \omega \}
		,\quad
		y \in \TT^d
		.
	\]

	We now proceed similarly as in the proof of \cite[Theorem~2.1\,(i)]{AM}: Given $x_0 \in \RR^d$, consider the particular
	(Gaussian) function $g = g_{x_0} \colon \RR^d \to \RR$ with
	\[
		g(x)
		=
		g_{x_0}(x)
		=
		\exp\biggl( -\frac{\abs{ x-x_0 }^2}{2}  \biggr)
		,\quad
		x \in \RR^d
		,
	\]
	the unitary Fourier transform of which is given by $(\cF g)(\xi) = \hat{g}(\xi) = e^{-ix_0\cdot\xi}g_0(\xi)$. With
	$h_t(\xi) = e^{-t\abs{\xi}^{2s}}g_0(\xi)$, $\xi \in \RR^d$, $t > 0$, we may choose $L > 0$ so large that
	\[
		C_{\omega,T} \int_0^T \norm{ \cF^{-1}h_t }_{L^2(B(0,L)^c)}^2 \,\dd t
		<
		\frac{1}{2}\,\norm{ e^{-T(-\Delta_x)^s}g }_{L^2(\RR^d)}^2
		.
	\]
	In light of
	\[
		e^{-t(-\Delta_x)^s}g
		=
		\cF^{-1}(\xi \mapsto e^{-t\abs{\xi}^{2s}}\hat{g}(\xi))
		=
		(\cF^{-1}h_t)(\cdot - x_0)
		,
	\]
	inserting $g = g_{x_0}$ into \eqref{eq:obsbaII} and \eqref{eq:obsbaSlice} and a change of variables then yield that
	\begin{align*}
		\frac{1}{2}\,\norm{ e^{-T(-\Delta_x)^s}g }_{L^2(\RR^d)}^2
		&\leq
		C_{\omega,T} \int_{\TT^d} \int_0^T \norm{ \cF^{-1}h_t }_{L^2((\omega_y-x_0)\cap B(0,L))}^2 \,\dd t\,\dd y\\
		&\leq
		C_{\omega,T} \int_0^T \norm{ \cF^{-1}h_t }_{L^\infty(\RR^d)}^2 \,\dd t\, \int_{\TT^d} \abs{\omega_y \cap B(x_0,L)} \,\dd y
	\end{align*}
	with
	\[
		\int_0^T \norm{ \cF^{-1}h_t }_{L^\infty(\RR^d)}^2 \,\dd t
		\leq
		\int_0^T \norm{h_t}_{L^1(\RR^d)}^2 \,\dd t
		\leq
		T\norm{g_0}_{L^1(\RR^d)}^2
		<
		\infty
		.
	\]
	Hence, for some $\theta \in (0,1]$ independent of $x_0$, we have
	\[
		\abs{ \omega \cap (B(x_0,L)\times\TT^d)}
		=
		\int_{\TT^d} \abs{ \omega_y\cap B(x_0,L) } \,\dd y
		\geq
		\theta\abs{B(x_0,L)}
		,
	\]
	which proves the claim.
\end{proof}

Parts of the statements of Theorems~\ref{thm:bagrushincont},~\ref{thm:critdiss} and~\ref{thm:lowdissba} can be proved
simultaneously. Here, we first focus on the positive results in Theorem~\ref{thm:bagrushincont}\,$(ii)\Rightarrow(i)$ and
Theorem~\ref{thm:critdiss}\,$(ii)$, which require some preparation.
Consider for $r>0$ the unitary transformation $M_{\gamma,r}$ in $L^2(\RR^d)$ defined by
\begin{equation}\label{eq:defMkr}
	M_{\gamma,r}g
	=
	r^{\frac{d}{2(\gamma+1)}} g(r^{\frac{1}{\gamma+1}}\cdot)
	,\quad
	g \in L^2(\RR^d)
	.
\end{equation}
With $H_{\gamma} = H_{\gamma;1}$, a straightforward computation shows that
\begin{equation}\label{eq:simitorus}
	(M_{\gamma,r})^* (H_{\gamma;r})^s M_{\gamma,r}
	=
	r^{\frac{2s}{\gamma+1}}(H_\gamma)^s
	,\quad
	r,s > 0
	.
\end{equation}
The latter allows, in particular, to obtain observability estimates for the operators $H_{\gamma;r}^s$, $r \geq 1$,
$s > 1/2$, simultaneously:

\begin{prop}\label{prop:interobsIII}
	Let $s>1/2$. Then, there exists a constant $K > 0$, depending only on $\gamma$, $s$, and the dimension $d$, such that for all
	$(\theta,L)$-thick sets $\omega \subset \RR^d$, $r\geq1$, $T > 0$, and $g \in L^2(\RR^d)$, we have
	\[
		\norm{ e^{-TH_{\gamma;r}^s}g }^2_{L^2(\RR^d)}
		\leq
		\frac{C_{\obs}}{T} \int_0^T \norm{ e^{-tH_{\gamma;r}^s}g }^2_{L^2(\omega)} \,\dd t
		,
	\]
	where the positive constant $C_{\obs} > 0$ is given by
	\begin{equation}\label{eq:obscst}
		C_{\obs}
		=
		K\biggl( \frac K{\theta} \biggr)^{K(1+rL^{1+\gamma})} \exp\biggl( \frac{K((1+L)\log(K/\theta))^{\frac{2s}{2s-1}}}{T^{\frac1{2s-1}}} \biggr)
		.
	\end{equation}
\end{prop}

\begin{proof}
	It follows from \eqref{eq:specweakly} with $a = 0$ and $\delta = 0$ that for every $(\theta,L)$-thick set
	$\omega \subset \RR^d$, we have
	\begin{equation}\label{eq:specweaklythick}
		\norm{f}^2_{L^2(\RR^d)}
		\leq
		\biggl( \frac C{\theta} \biggr)^{C(1+L^{1+\gamma}+L\sqrt{\lambda}+\log(1+\lambda))} \norm{f}^2_{L^2(\omega)}
		,\quad
		f \in \cE_{\lambda,\gamma,1}
		,\
		\lambda \geq 0
		,
	\end{equation}
	with a constant $C > 0$ depending only on $\gamma$ and the dimension $d$.

	Let us now fix some $r \geq 1$ and a $(\theta,L)$-thick set $\omega \subset \RR^d$. In light of the similarity relation
	\eqref{eq:simitorus}, we clearly have
	\[
		(M_{\gamma,r})^*\cE_\lambda(H_{\gamma;r}^s)
		\subset
		\cE_\lambda(r^{\frac{2s}{\gamma+1}}H_\gamma^s)
		=
		\cE_{\lambda/r^{2s/(\gamma+1)}}(H_\gamma^s)
		=
		\cE_{\lambda^{1/s}/r^{2/(\gamma+1)},\gamma,1}
		.
	\]
	Moreover, one easily checks that the set $\tilde{\omega} := r^{1/(\gamma+1)}\omega$ is $(\theta,r^{1/(\gamma+1)}L)$-thick. We
	therefore deduce from \eqref{eq:specweaklythick} that for all $\lambda \geq 0$ and $f \in \cE_\lambda(H_{\gamma;r}^s)$, we have
	\begin{align*}
		\norm{f}_{L^2(\RR^d)}^2
		=
		\norm{(M_{\gamma,r})^*f}_{L^2(\RR^d)}^2
		&\leq
		\biggl( \frac{C}{\theta} \biggr)^{C(1+rL^{1+\gamma}+L\lambda^{\frac1{2s}} + \log(1+r^{-\frac2{\gamma+1}}\lambda^{\frac1s}))}
			\norm{(M_{\gamma,r})^*f}_{L^2(\tilde{\omega})}^2\\
		&\le
		\biggl( \frac{C}{\theta} \biggr)^{C(1+rL^{1+\gamma}+(1+L)\lambda^{\frac1{2s}})} \norm{f}_{L^2(\omega)}^2
		,
	\end{align*}
	since $r \geq 1$ and, thus,
	$\log(1+r^{-\frac2{\gamma+1}}\lambda^{\frac1s}) \leq \log(1+\lambda^{\frac1s})\leq\lambda^{\frac1{2s}}$. The latter can be
	rewritten as
	\[
		\norm{f}_{L^2(\RR^d)}^2
		\leq
		d_0e^{d_1\lambda^{\frac1{2s}}}\norm{f}_{L^2(\omega)}^2
		,\quad
		f \in \cE_\lambda(H_{\gamma;r}^s) 
		,
	\]
	with
	\[
		d_0
		=
		\biggl( \frac{C}{\theta} \biggr)^{C(1+rL^{1+\gamma})}
		\quad\text{ and }\quad
		d_1
		=
		C(1+L)\log\biggl( \frac{C}{\theta} \biggr)
		.
	\]
	Theorem~\ref{thm:obs} then implies that there exist universal positive constants $c_1,c_2,c_3>0$ such that for all $T>0$ and
	$g \in L^2(\RR^d)$, we have
	\[
		\norm{e^{-TH_{\gamma;r}^s}g}_{L^2(\RR^d)}^2
		\leq
		\frac{C_\obs}T \int_0^T \norm{e^{-tH_{\gamma;r}^s}g}_{L^2(\omega)}^2 \,\dd t
		,
	\]
	with $C_\obs = C_\obs(\omega,T,r)$ given by
	\[
		C_\obs
		=
		c_1d_0(2d_0+1)^{c_2} \exp\biggl( \frac{c_3d_1^{\frac{2s}{2s-1}}}{T^{\frac1{2s-1}}} \biggr)
		.
	\]
	It only remains to observe that there exists another positive constant $c_4 > 0$, depending only on the dimension $d$, such that
	\[
		d_0(2d_0+1)^{c_2}
		\leq
		\biggl( \frac{c_4}{\theta} \biggr)^{c_4(1+rL^{1+\gamma})}
		.
	\]
	This ends the proof of Proposition~\ref{prop:interobsIII} upon a suitable choice of the constant $K$.
\end{proof}

\begin{proof}[Proof of Theorem~\ref{thm:bagrushincont}\,$(ii)\Rightarrow(i)$ and Theorem~\ref{thm:critdiss}\,$(ii)$]
	Let $\omega \subset \RR^d$ be a $(\theta,L)$-thick set. We have to show that whenever $T \geq T^*$, with some time
	$T^* \geq 0$ depending on $\theta$ and $L$ that is to be determined, there exists a constant $C_{\omega,T} > 0$
	such that for all $g \in L^2(\RR^d\times\TT^d)$ we have
	\begin{equation}\label{eq:ponctobsII}
		\norm{ e^{-T(-\Delta_\gamma)^s}g }^2_{L^2(\RR^d\times\TT^d)}
		\leq
		C_{\omega,T} \int_0^T \norm{ e^{-t(-\Delta_\gamma)^s}g }^2_{L^2(\omega\times\TT^d)} \,\dd t
		.
	\end{equation}
	To this end, we first observe from \eqref{eq:GrushinTransformed}, Fubini's theorem, and Parseval's identity that for every
	measurable set $\Omega \subset \RR^d$ and all $t > 0$ and $g \in L^2(\RR^d\times\TT^d)$ we have
	\[
		\norm{ e^{-t(-\Delta_\gamma)^s}g }^2_{L^2(\Omega\times\TT^d)}
		=
		\sum_{n\in\ZZ^d} \norm{ e^{-tH_{\gamma;n}^s} \hat{g}_n }^2_{L^2(\Omega)}
		,
	\]
	where we write $H_{\gamma;n} = H_{\gamma;\abs{n}}$ for every $n \in \ZZ^d$. Inserting the latter into both sides of
	\eqref{eq:ponctobsII}, once with $\Omega = \RR^d$ and $t = T$ and once with
	$\Omega = \omega$, we immediately infer by Fubini's theorem that it suffices to show that
	\begin{equation}\label{eq:obsFourier}
		\norm{ e^{-TH_{\gamma;n}^s}g }^2_{L^2(\RR^d)}
		\leq
		C_{\omega,T} \int_0^T \norm{ e^{-tH_{\gamma;n}^s}g }^2_{L^2(\omega)} \,\dd t
		,\quad
		g \in L^2(\RR^d)
		,\
		n \in \ZZ^d
		,
	\end{equation}
	with a constant $C_{\omega,T} > 0$ not depending on $n$. Here, for $n = 0$, the operator $H_{\gamma;0}^s$ reduces to the
	fractional Laplacian $(-\Delta_x)^s$ on $\RR^d$. Corresponding observability estimates from thick sets are well known in the
	literature, see, e.g., \cite[Theorem~4.10]{NakicTTV-20} or \cite[Theorem~1.12]{AB}. It is therefore sufficient to focus on the
	case $\abs{n} \geq 1$. Here, on the one hand, we deduce from \eqref{eq:simitorus} that
	\[
		\norm{ e^{-tH_{\gamma;n}^s}g }_{L^2(\RR^d)}
		\leq
		e^{-t\lambda_\gamma^s\abs{n}^{\frac{2s}{{1+\gamma}}}} \norm{g}_{L^2(\RR^d)}
		,\quad
		g \in L^2(\RR^d)
		,\
		t \geq 0
		,
	\]
	where $\lambda_{\gamma}>0$ again denotes the smallest eigenvalue of the anharmonic oscillator $H_{\gamma}$.
	This implies, in particular, that
	\[
		\norm{ e^{-TH_{\gamma;n}^s}g }^2_{L^2(\RR^d)}
		\leq
		e^{-T\lambda_\gamma^s\abs{n}^{\frac{2s}{{1+\gamma}}}} \norm{e^{-(T/2)H_{\gamma;n}^s}g}^2_{L^2(\RR^d)}
		,\quad
		g \in L^2(\RR^d)
		.
	\]	
	On the other hand, it follows from Proposition \ref{prop:interobsIII} that for all $n \in \ZZ^d\setminus\{0\}$ and
	$g \in L^2(\RR^d)$, we have
	\[
		\norm{ e^{-(T/2)H_{\gamma;n}^s}g }^2_{L^2(\RR^d)}
		\leq
		\frac{2C_{\obs}}{T} \int_0^{T/2} \norm{ e^{-tH_{\gamma;n}^s}g }^2_{L^2(\omega)} \,\dd t
		,
	\]
	where $C_{\obs} = C_{\obs}(\omega,T/2,\vert n\vert)$ is given by \eqref{eq:obscst} with $T$ replaced by $T/2$. Combining these
	two estimates, we therefore obtain that for all $n \in \ZZ^d \setminus \{0\}$ and $g\in L^2(\RR^d)$,
	\begin{multline*}
		\norm{ e^{-TH_{\gamma;n}^s}g }^2_{L^2(\RR^d)}
		\leq
		\exp\bigl( K\abs{n}L^{1+\gamma}\log(K/\theta)-\abs{n}^{\frac{2s}{1+\gamma}}\lambda_\gamma^sT \bigr) \\
		\times K\biggl( \frac K{\theta} \biggr)^K\exp\biggl( \frac{K((1+L)\log(K/\theta))^{\frac{2s}{2s-1}}}{T^{\frac1{2s-1}}} \biggr)
			\frac2T \int_0^{T/2} \norm{ e^{-tH_{\gamma;n}^s}g }^2_{L^2(\omega)} \,\dd t
		.
	\end{multline*}
	This shows \eqref{eq:obsFourier}, provided that
	\[
		\sup_{\abs{n}\geq1}\exp\bigl( K\abs{n}L^{1+\gamma}\log(K/\theta)-\abs{n}^{\frac{2s}{1+\gamma}}\lambda_\gamma^sT \bigr)
		<
		+\infty.
	\]
	The latter is the case for every $T > 0$ if $s > (1+\gamma)/2$, which proves the implication $(ii)\Rightarrow(i)$ in
	Theorem~\ref{thm:bagrushincont}, and if $s = (1+\gamma)/2$, it requires
	\[
		T
		\geq
		T^*
		:=
		K\lambda_\gamma^{-s}L^{1+\gamma}\log(K/\theta)
		,
	\]
	as claimed in Theorem~\ref{thm:critdiss}\,$(ii)$.
\end{proof}

We now finally turn to the negative null-controllability results for the equation~\eqref{eq:bagrushin}.

\begin{proof}[Proof of Theorem~\ref{thm:critdiss}\,$(i)$ and Theorem~\ref{thm:lowdissba}]
	Let $\omega \subset \RR^d \times \TT^d$ be a measurable set satisfying the geometric condition $\overline\omega\cap\{x=0\} = \emptyset$.
	We assume that for some positive time $T > 0$ there exists a positive constant $C_{\omega,T}>0$ such that for all functions
	$g\in L^2(\RR^d \times \TT^d)$ we have the observability estimate
	\begin{equation}\label{eq:obsBaGrushin}
		\norm{ e^{-T(-\Delta_\gamma)^s}g }^2_{L^2(\RR^d \times \TT^d)}
		\leq
		C_{\omega,T} \int_0^T \norm{ e^{-t(-\Delta_\gamma)^s}g }^2_{L^2(\omega)} \,\dd t
		.
	\end{equation}
	Let $\psi_\gamma \in L^2(\RR^d)$ be a normalized eigenfunction for the anharmonic oscillator $H_\gamma$ corresponding to the
	smallest eigenvalue $\lambda_\gamma>0$. For each $n \in \ZZ^d \setminus \{0\}$, consider the function
	$g_n \in L^2(\RR^d\times \TT^d)$ given by
	\begin{equation}\label{eq:defgn}
		g_n(x,y)
		=
		e^{in\cdot y} (M_{\gamma,n} \psi_\gamma)(x)
		,\quad
		(x,y) \in \RR^d \times \TT^d
		,
	\end{equation}
	where the isometry $M_{\gamma,n} = M_{\gamma,\abs{n}}$ in $L^2(\RR^d)$ is defined as in \eqref{eq:defMkr}. In light of the
	similarity relation \eqref{eq:simitorus}, it is then clear that
	$(-\Delta_\gamma)g_n = \abs{n}^{\frac 2{1+\gamma}}\lambda_\gamma g_n$ as well as
	\[
		\norm{g_n}_{L^2(\RR^d \times \TT^d)}
		=
		1
		\quad\text{ and }\quad
		\norm{g_n}_{L^2(\omega)}
		=
		\norm{\psi_\gamma}_{L^2(\omega_n)}
		,
	\]
	where
	\[
		\omega_n
		=
		\{ (\abs{n}^{1/(1+\gamma)}x , y) \colon (x,y) \in \omega \}
		,
	\]
	and where $\psi_\gamma$ is interpreted as a function in $L^2(\RR^d \times \TT^d)$ that is constant with respect to the
	$\TT^d$-variable.
	The observability estimate \eqref{eq:obsBaGrushin} applied to $g = g_n$ therefore implies that
	\begin{equation}\label{eq:obsBaGrushingn}
		e^{-2\abs{n}^{\frac{2s}{1+\gamma}}\lambda_\gamma^sT}
		\leq
		C_{\omega,T} \int_0^T e^{-2\abs{n}^{\frac{2s}{1+\gamma}}\lambda_\gamma^st} \norm{\psi_\gamma}^2_{L^2(\omega_n)}\,\dd t
		\leq
		TC_{\omega,T} \norm{\psi_\gamma}^2_{L^2(\omega_n)}
		.
	\end{equation}

	Using the classical, more precise Agmon estimate for the anharmonic oscillator $H_\gamma$ mentioned in part (2) of
	Remark~\ref{rk:agmon}, for every $\varepsilon\in(0,1)$ we have
	\[
		\norm{ e^{\frac{\varepsilon\abs{x}^{1+\gamma}}{1+\gamma}}\psi_\gamma }_{L^2(\RR^d)}
		\leq
		c_{\varepsilon,\gamma}
		,
	\]
	where $c_{\varepsilon,\gamma}>0$ is a positive constant depending only on $\varepsilon$, $\gamma$, and the dimension $d$.
	Thus, with $L := \dist(0,\omega)$, for each $n \in \ZZ^d \setminus \{0\}$ we have
	\begin{equation}\label{eq:boundpsik}
		\norm{\psi_\gamma}_{L^2(\omega_n)}
		=
		\norm{ e^{-\frac{\varepsilon\abs{x}^{1+\gamma}}{1+\gamma}} e^{\frac{\varepsilon\abs{x}^{1+\gamma}}{1+\gamma}}
			\psi_\gamma }_{L^2(\omega_n)}
		\leq
		c_{\varepsilon,\gamma}e^{-\frac{\varepsilon\abs{n} L^{1+\gamma}}{1+\gamma}}
		.
	\end{equation}
	Inserting the latter into \eqref{eq:obsBaGrushingn}, we deduce for each $n \in \ZZ^d \setminus \{0\}$ that
	\begin{equation}\label{eq:obsBaGrushingnNec}
		1
		\leq
		TC_{\omega,T}c_{\varepsilon,\gamma}^2
			\exp\biggl( 2\abs{n}^{\frac{2s}{1+\gamma}}\lambda_\gamma^sT - \frac{2\varepsilon\vert n\vert L^{1+\gamma}}{1+\gamma} \biggr)
		.
	\end{equation}
	Now, if $0<s<(1+\gamma)/2$ or if $s=(1+\gamma)/2$ and $0 < T < (\varepsilon/(1+\gamma))(L/\sqrt{\lambda_\gamma})^{\gamma+1}$,
	then
	\[
		\exp\biggl( 2\abs{n}^{\frac{2s}{1+\gamma}}\lambda_\gamma^sT - \frac{2\varepsilon\vert n\vert L^{1+\gamma}}{1+\gamma} \biggr)
		\underset{\vert n\vert\rightarrow+\infty}{\rightarrow}
		0
		,
	\]
	which contradicts the estimate \eqref{eq:obsBaGrushingnNec}. This ends the proof of Theorem~\ref{thm:lowdissba} and,
	after letting $\varepsilon\rightarrow1^-$, also of the one of Theorem~\ref{thm:critdiss}\,$(i)$.
\end{proof}

\begin{rk}\label{rk:schrogrushin}
	It is worth to note that the Schr\"odinger-type equation corresponding to the fractional Baouendi-Grushin operator, that is,
	the equation \eqref{eq:sbagrushin}, is for no $s>0$ and at no time $T>0$ null-controllable from a control support $\omega$
	satisfying the geometric condition $\overline\omega\cap\{x=0\}=\emptyset$. Indeed, assume to the contrary that there exists a
	positive constant $C_{\omega,T} > 0$ such that for all $g\in L^2(\RR^d \times \TT^d)$ we have the observability estimate
	\begin{equation}\label{ex:schrogrushin}
		\norm{g}^2_{L^2(\RR^d \times \TT^d)}
		\leq
		C_{\omega,T} \int_0^T \norm{ e^{it(-\Delta_\gamma)^s}g }^2_{L^2(\omega)}\,\dd t
		.
	\end{equation}
	Inserting again the function $g_n$ defined in \eqref{eq:defgn} and using the estimate \eqref{eq:boundpsik}, we deduce that for
	all $n \in \ZZ^d \setminus \{0\}$ we have
	\[
		1
		\leq
		C_{\omega,T} Tc_{\varepsilon,\gamma}^2 e^{-\frac{2\varepsilon\abs{n} L^{1+\gamma}}{1+\gamma}}
			\underset{\abs{n}\rightarrow+\infty}{\rightarrow}0
		.
	\]
	Hence, the estimate \eqref{ex:schrogrushin} can never hold for all $g \in L^2(\RR^d \times \TT^d)$ simultaneously.
\end{rk}

\appendix

\section{Asymptotic bounds on the smallest eigenvalue of anharmonic oscillators}\label{sec:asymptotics}

In this appendix, we prove a two-sided asymptotics as $k \to +\infty$ for the smallest eigenvalue $\lambda_k$ of the
anharmonic oscillator $H_k = H_{k,1} = -\Delta + \abs{x}^{2k}$, $k \in \NN \setminus \{0\}$, in $L^2(\RR^d)$ equipped with its
maximal domain. This is a key ingredient for Example~\ref{ex:example} in the main part of the manuscript.

\begin{lem}\label{lem:asymptotics}
	For fixed $\varepsilon > 0$, the two-sided bound
	\begin{equation}\label{eq:eigenv}
		\frac{\lambda_D}{(1+\varepsilon)^2} + o(1)
		\leq
		\lambda_{k}
		\leq
		\lambda_D + \int_{B(0,1)} \abs{x}^{2k} \abs{\psi_D(x)}^2 \,\dd x
	\end{equation}
	holds, where the lower bound is to be understood as $k \to +\infty$, and where $\lambda_D>0$ denotes the smallest eigenvalue of
	the Dirichlet Laplacian on the canonical Euclidean unit ball $B(0,1)$ in $\RR^d$ and $\psi_D$ is an associated normalised
	eigenfunction.
\end{lem}

\begin{proof}
	The upper bound in \eqref{eq:eigenv} follows immediately from the standard min-max principle. Let us therefore focus on the
	lower bound. To this end, fix $\varepsilon>0$ and observe that for all $x \notin B(0,1+\varepsilon)$ we have
	$\abs{x}^{2k} \geq (1+\varepsilon)^{2k}$. This gives
	\[
		-\Delta + (1+\varepsilon)^{2k} \bmone_{B(0,1+\varepsilon)^c}
		\leq
		H_k
		,
	\]
	in the sense of quadratic forms, and it follows from the min-max principle that
	\[
		\min\spec(-\Delta + (1+\varepsilon)^{2k} \bmone_{B(0,1+\varepsilon)^c})
		\leq
		\lambda_k
		.
	\]
	By a standard scaling argument, the operator $-\Delta + (1+\varepsilon)^{2k} \bmone_{B(0,1+\varepsilon)^c}$ is unitarily
	equivalent to $(1+\varepsilon)^{-2}(-\Delta + (1+\varepsilon)^{2(1+k)}\bmone_{B(0,1)^c})$. Moreover, by the theory of the
	large coupling limit \cite{Bruneau, Demuth, Simon}, the spectrum of $-\Delta + M\bmone_{B(0,1)^c}$ converges to the one of
	the Dirichlet Laplacian on $B(0,1)$ as $M$ goes to infinity. More specifically, if follows from \cite{Demuth} that
	$-\Delta + M\bmone_{B(0,1)^c}$ converges to the Dirichlet Laplacian on $B(0,1)$ in norm resolvent sense as
	$M\rightarrow+\infty$, so that indeed
	\[
		\min\spec(-\Delta + M\bmone_{B(0,1)^c})
		=
		\lambda_D + o(1)
		\quad
		\text{as $M \to +\infty$}
		.
	\]
	Applying this result with $M = (1+\varepsilon)^{2(1+k)}$ together with the unitary equivalence mentioned above then proves
	the lower bound in \eqref{eq:eigenv}. This completes the proof.
\end{proof}

\begin{rk}\label{rk:asymptotics}
	Since the eigenfunction $\psi_D$ is radially symmetric, one may introduce a new function $\varphi_D \in C^\infty([0,1])$ with
	$\abs{\psi_D(x)}^2 = \varphi_D(\abs{x})$. Using polar coordinates, we then obtain
	\[
		\int_{B(0,1)} \abs{x}^{2k} \abs{\psi_D(x)}^2 \,\dd x
		=
		\vert\mathbb S^{d-1}\vert\, \int_0^1 r^{2k+d-1} \varphi_D(r) \,\dd r
		.
	\]
	Now, successive integration by parts in the last integral gives for all $N \geq 1$ that, as $k \to +\infty$,
	\[
		\int_0^1 r^{2k+d-1} \varphi_D(r) \,\dd r
		=
		\sum_{j=1}^{N-1}\frac{\varphi_D^{(j-1)}(1)}{(2k+d-1+j)^j} + \cO\biggl( \frac1{k^N} \biggr)
		.
	\]
\end{rk}

From Lemma~\ref{lem:asymptotics} and Remark~\ref{rk:asymptotics} and considering $\varepsilon \to 0^+$ in \eqref{eq:eigenv}, we
immediately obtain the following result.

\begin{cor}\label{cor:asymptotics}
	We have $\lambda_k \to \lambda_D$ as $k \to +\infty$.
\end{cor}

\end{document}